\documentclass[a4paper]{amsart}
\usepackage[utf8]{inputenc}
\usepackage[numbers,sort]{natbib}
\usepackage{amsmath}
\usepackage{color}
\usepackage{enumerate}
\usepackage{algorithm}
\usepackage[noend]{algpseudocode}
\usepackage[colorlinks=true,linkcolor=blue,citecolor=blue]{hyperref}
\usepackage{graphicx}
\usepackage{csquotes}
\usepackage[a4paper,margin=1.2in]{geometry}

\newcommand{\E}{\mathbb{E}}    
\renewcommand{\P}{\mathbb{P}}  
\newcommand{\R}{\mathbb{R}}    
\newcommand{\uarg}{\,\cdot\,}  
\newcommand{\M}{\mathbf{M}}    
\newcommand{\Mm}{\underline{M}_m}
\newcommand{\osc}{\operatorname{osc}} 
\newcommand{\ud}{\mathrm{d}}   
\renewcommand{\vec}[1]{\mathbf{#1}} 
\newcommand{\Var}{\mathrm{Var}}
\newcommand{\Y}{S}
\newcommand{\I}{\mathbb{I}}

\newtheorem{theorem}{Theorem}[section]

\newtheorem{lemma}[theorem]{Lemma}

\newtheorem*{assumption*}{Assumption}
\newtheorem*{definition*}{Definition}

\theoremstyle{definition}
\newtheorem{definition}[theorem]{Definition}

\newtheorem{example}[theorem]{Example}
\newtheorem{remark}[theorem]{Remark}

\begin{document}

\begin{abstract}
  We study the forgetting properties of the particle filter when its state --- the collection of particles --- is regarded as a Markov chain. Under a strong mixing assumption on the particle filter's underlying Feynman--Kac model, we find that the particle filter is exponentially mixing, and forgets its initial state in $O(\log N )$ `time', where $N$ is the number of particles and time refers to the number of particle filter algorithm steps, each comprising a selection (or resampling) and mutation (or prediction) operation. We present an example which shows that this rate is optimal.  In contrast to our result, available results to date are extremely conservative, suggesting $O(\alpha^N)$ time steps are needed, for some $\alpha>1$, for the particle filter to forget its initialisation. We also study the {\it conditional} particle filter (CPF) and  extend our forgetting result to this context. We establish a similar conclusion, namely, CPF is exponentially mixing and forgets its initial state in $O(\log N)$ time. To support this analysis, we establish new time-uniform $L^p$ error estimates for CPF, which can be of independent interest. We also establish new propagation-of-chaos type results using our proof techniques, discuss implications to couplings of particle filters and an application to processing out-of-sequence measurements.
\end{abstract}

\title{On the forgetting of particle filters}

\author{Joona Karjalainen}
\address{Joona Karjalainen, Department of Mathematics and Statistics, University of
Jyväskylä, P.O.Box 35, FI-40014 University of Jyväskylä, Finland}
\author{Anthony Lee}
\address{Anthony Lee, School of Mathematics, University of Bristol,
  Woodland Rd, Bristol BS8 1UG, United Kingdom}
\author{Sumeetpal S. Singh}
  \address{Sumeetpal S.~Singh, 
NIASRA, School of Mathematics and Physics, University of Wollongong, NSW 2522, Australia}
\author{Matti Vihola}
\address{Matti Vihola, Department of Mathematics and Statistics, University of
Jyväskylä, P.O.Box 35, FI-40014 University of Jyväskylä, Finland}

\maketitle
\section{Introduction}

Particle filters \citep{gordon-salmond-smith}, or sequential Monte Carlo methods, have become a workhorse for applied non-linear stochastic filtering, and more generally, the go-to approach for statistical inference in state-space models; for example, see \citep{chopin-papaspiliopoulos,cappe-moulines-ryden} for a general introduction to particle filtering and its applications in various scientific fields.
The particle filter algorithm propagates $N\ge 2$ random samples, or `particles' $\vec{X}_k = (X_k^1,\ldots,X_k^N)$ through time $k=1,2,\ldots$, which form an empirical approximation of the time-evolving sequence of filtering probability distributions $\pi_k$. In particular, for a test function $\phi$, a weighted average of the particles of the form $\sum_{i=1}^N W_k^i \phi(X_k^i)$ approximates the integral of $\phi$ with respect to $\pi_k$; see the discussion in Section \ref{sec:pf}.

The particle filter was largely popularised by the seminal work of Gordon et al.~\citep{gordon-salmond-smith} in applied non-linear filtering. Since then, over the past decades, there have been numerous advancements in the design of more efficient particle filtering algorithms, in theoretical studies, and its use in new applications; see the textbooks 
Cappé et al.~\citep{cappe-moulines-ryden}; Chopin and Papaspiliopoulos \citep{chopin-papaspiliopoulos}; Del Moral \citep{del2004feynman} for a detailed account. 
The particle filter approximation is asymptotically accurate for large $N$, and its behaviour is theoretically well understood. The collection of theoretical results available includes central limit theorems \citep{delmoral-guionnet, chopin, douc-moulines}; large deviations \citep{delmoral-guionnet-98};  time-uniform $L^p$ errors \citep{del2004feynman}; propagation of chaos \cite{delmoral-doucet-peters}; and convergence of $U$-statistics \citep{delmoral-patras-rubenthaler}.
Available theoretical results are not only for the non-linear filter but have also been extended to various other estimates of interest. For example, the smoother \citep{delmoral-doucet-singh, douc-etal-2011}; the model evidence \citep{cerou-delmoral-guyader,whiteley-stability}; the gradient of the filter \citep{del2015uniform};  and the particle filter's variance \citep{chan-lai,lee-whiteley, olsson-douc}.

The theoretical results above are mainly about approximation of integrals with particles, whereas we focus on a natural question of a different nature.
We regard the particle filter's vector-valued state $\vec{X}_k$ as an inhomogeneous Markov chain $(\vec{X}_k)_{k\ge 1}$ and study its forgetting properties.
Unlike some of the results for empirical measures, we do not concentrate on the asymptotic regime $N\to\infty$, but we do investigate how forgetting depends on $N$.
We prove that, under a strong mixing condition, the Markov chain $(\vec{X}_k)$ is exponentially forgetting, and forgets its state in $O(\log N)$ time (Theorem \ref{thm:pfforget}).
An example, which we present in Section \ref{sec:example}, shows that this rate is optimal.

Particle filter forgetting in our sense has been considered before, at least in \cite{tadic-doucet-recursive}. The work of \cite{tadic-doucet-recursive} studies the bias of gradient-based recursive maximum likelihood estimation for  state-space models, when a particle filter is used to approximate the required gradients; this is a popular approach for calibrating state-space models \cite{poyiadjis-etal, kantas-etal, schon-etal}. The forgetting of the particle filter plays a crucial role in establishing the bias of the maximum likelihood estimate, which is induced by the particle filter's approximation of the gradients (e.g., see the series of Lemmas 5.1 to 5.5 in \cite{tadic-doucet-recursive}).
They establish an exponential forgetting result (restated as Lemma \ref{lem:easybound} here), which guarantees a Dobrushin coefficient $1-O(\epsilon^N)$, or forgetting in $O(\epsilon^{-N})$ time, which turns out to be extremely loose for large $N$.
In contrast, assuming similar strong mixing conditions as in \cite{tadic-doucet-recursive}, which are frequently used in the literature --- including  many of the theoretical results cited above --- we establish a uniform Dobrushin coefficient for the composition of $O(\log N)$ particle filter steps (Lemma \ref{lem:tvuniform}). 

We then turn to investigate the conditional particle filter (CPF) \cite{andrieu-doucet-holenstein}, which is firmly established as an important algorithm for Bayesian estimation of general state-space models. The CPF algorithm is similar to the particle filter, except for a fixed `reference' path, but its purpose is very different: it defines a Markov transition targeting the smoothing distribution. 
We prove a forgetting result for the CPF (Theorem \ref{thm:cpfforget}), relying on similar ingredients as for the particle filter, and a new time-uniform $L^p$ error bound for the CPF (Theorem \ref{thm:cpfstability}), which is of independent interest. These results are necessary for the analysis of a coupling CPF algorithm and lead to a characterisation of the rate of convergence for CPFs with backward sampling \citep{karjalainen-lee-singh-vihola-mixing}.
In fact, this rate of convergence result was a major motivation to investigate the broader question of forgetting properties of particle filters: standard results such as uniform-in-time error and ideal filter forgetting (Lemmas \ref{lem:fkcontract} and \ref{lem:pfstability}) were not sufficient for the analysis.

Section \ref{sec:pochaos} discusses `propagation of chaos' bounds, which measure how similar a subset of $q$ particles is in law to $q$ independent draws from the ideal filter. The techniques of the present paper lead to new bounds (Theorems \ref{thm:uniform-poc} and \ref{thm:forget-poc}), which are useful when $q$ is large, unlike existing bounds in the literature. In Sections \ref{sec:couplings} and \ref{sec:oos}, we discuss implications of our findings to implementable particle filter couplings and show how they can be used to establish a new method for the out-of-sequence measurements problem in non-linear filtering. In Section \ref{sec:qmarginals}, we prove the non-forgetting result stated in Section \ref{sec:example} and extend it to the $q$-marginals of the particle system. Section \ref{sec:discussion} concludes with a discussion about the findings, the necessity of the assumptions, and future research topics.

\section{Preliminaries and particle filter definition}
\label{sec:pf}

We follow the (now common) convention of \cite{del2004feynman} and define a sequence of probability measures $\{\eta_k(\ud x)\}_{k \geq 1}$ on a general state-space $E$, which we assume to be Polish, through a \emph{Feynman--Kac} (FK) model. 
The particle filter, which is defined in Algorithm \ref{alg:pf} below, is a 
Monte Carlo approximation of these probability measures.

Before defining the FK model, we declare the adopted mathematical notation. Here and throughout this paper, we implicitly assume test sets and functions to be measurable. Let $M$ be a Markov kernel on $E$, and let $\mu$ and $\nu$ be probability measures on $E$. Given a real-valued measurable function $\phi:E\to\R$ and $A\subset E$, let $\mu(\phi) = \int \phi \ud \mu = \int \phi(x) \mu(\ud x)$, $(\mu M)(A) = \int \mu(\ud x) M(x,A)$ and $(M \phi)(x) = \int M(x,\ud y)\phi(y)$.
We denote the total variation distance between $\mu$ and $\nu$ by
$$
\left \| \mu - \nu \right \|_{\rm TV} = 
\sup_{\| \phi\|_\infty \leq 1/2} | \mu(\phi) - \nu(\phi)| = 
\sup_{\osc(\phi) \leq 1} | \mu(\phi) - \nu(\phi)|,$$
where $\|\phi\|_\infty = \sup_{x} |\phi(x)|$,
and $\mathrm{\osc}(\phi) = \sup_x \phi(x) - \inf_x \phi(x)$. We will occasionally use the well-known coupling inequality $\left \| \mu - \nu \right \|_{\rm TV} \leq \P (X \neq Y)$ for any $X \sim \mu$, $Y \sim \nu$ \cite[][eq. (2.6)]{lindvall2002lectures}, and in particular, 
$$
\left \| \mu - \nu \right \|_{\rm TV} = \inf \{ \P (X \neq Y) \; : \; X \sim \mu, \; Y \sim \nu \}.
$$

The composition of Markov kernels $M_1,M_2$ is a Markov kernel defined by $(M_1 M_2)(x,A) = \int M_1(x,\ud y) M_2(y, A)$, and more generally, we denote iterated compositions as products.
We denote the $k$-fold product measure of $\mu$ by $\mu^{\otimes k}$. 
If $\mu$ and $\nu$ have densities $p$ and $q$ with respect to a common dominating measure $\lambda$, we note that
$\left \| \mu - \nu \right \|_{\rm TV} = \frac{1}{2} \int \left | p - q \right | {\rm d}\lambda = 1 - \int \min(p,q) {\rm d} \lambda$.
The $L^p$ norm of a random variable $X$ is denoted by $\left \| X \right \|_{p} = \E(|X|^p)^{1/p}$.

Many of our main results involve bounding Dobrushin (or Doeblin) coefficients of appropriate sequences of Markov kernels.
The equivalence motivating the following definition can be found in, e.g., \cite[][Section~18.2]{douc2018markov}.
\begin{definition}
For a Markov kernel $M$ on $E$, we denote its Dobrushin contraction coefficient by
$$\beta_{\rm TV}(M) = \sup_{x, y\in E} \left\| M(x,\uarg) - M(y, \uarg) \right\|_{\mathrm{TV}} 
= \sup_{\mu,\nu \in \mathcal{P}, \, \mu\neq \nu } \frac{\left\| \mu M - \nu M \right\|_{\mathrm{TV}}}{\left\| \mu - \nu \right\|_{\mathrm{TV}}},$$
where $\mathcal{P}$ is the set of probability measures on $E$.
\end{definition}
Note that the latter form implies sub-multiplicativity: for two Markov kernels $M_1$ and $M_2$, $\beta_{\rm TV}(M_1 M_2) \leq \beta_{\rm TV}(M_1) \beta_{\rm TV}(M_2)$.

A Feynman--Kac model $\big(\eta_0,(M_k)_{k\ge 1},(G_k)_{k\ge 0}\big)$ on 
$E$  consists of an initial probability measure $\eta_0$ on $E$, a sequence of transition probability kernels $M_k$ on $E$ and a sequence of non-negative functions $G_k:E\to[0,\infty)$ which are called the `potential functions'. 
We denote the mappings from $\mathcal{P}$ to $\mathcal{P}$ stemming from the FK model as follows: 
\begin{gather*}
  \Psi_k(\mu)({\rm d}x) \;=\; \frac{G_k(x) \mu({\rm d}x)}{\mu(G_k)}, \;\;\; k \geq0, \\
  \Phi_{k+1}(\mu) \;=\; \Psi_{k}(\mu)M_{k+1}, \;\;\; k \geq 0, \qquad
   \Phi_{k,n} \;=\; \Phi_n \circ \cdots \circ \Phi_{k+1}, \;\;\; 0 \leq k < n, 
\end{gather*}
and $\Phi_{n,n}$ is the identity mapping for all $n\ge 0$.
Using these mappings, define
\[
    \eta_k := \Phi_{0,k}(\eta_0).
\]
In the engineering and statistics literature, $\eta_k$ is known as the `predictor' and represents the Bayesian posterior distribution --- in the context of probabilistic time-series models  --- of the hidden state of the time-series at time $k$, given observations up to time $k-1$ \cite[e.g.][Section 12.6.3]{del2004feynman}. The `filter' is the updated $\eta_k$ that incorporates the observation at time $k$, that is $\pi_k := \Psi_k(\eta_k)$. One of the most popular methods to approximate $\eta_k$ is by using a particle filter.

\newcommand{\ind}[1]{{#1}}
\begin{algorithm}
  \caption{\textsc{ParticleFilter}$(\eta_0, M_{1:T}, G_{0:T-1}, N)$}
  \label{alg:pf} 
\begin{algorithmic}[1]
\State
Draw $X_0^{\ind{i}} \sim \eta_0(\uarg)$ for $i \in \{1,\ldots,N\}$ 
\For{$k=0,\ldots,T-1$}
\State Draw $A_{k}^{\ind{i}} \sim \mathrm{Categ}_{(1:N)}\big(
G_{k}(X_{k}^{\ind{1}}),\ \dots, G_{k}(X_{k}^{\ind{N}})\big)$ 
for $i \in \{1,\ldots,N\}$
\State Draw $X_{k+1}^{\ind{i}} \sim M_{k+1}(X_{k}^{\ind{A_{k}^{\ind{i}}}}, \uarg)$ for $i \in \{1,\ldots,N\}$  
\EndFor
\end{algorithmic}
\end{algorithm}
Algorithm \ref{alg:pf} summarises the particle filter in pseudo-code, where we denote by $a{:}b$ the integers from $a$ to $b$, and use this in indexing, so that $M_{1:T} = (M_1,\ldots,M_T)$. All random variables generated in the algorithm are conditionally independent of earlier generated variables given the random variables appearing in their distribution,  and ${\rm Categ}_{(n:m)}(w_{n:m})$ stands for the categorical distribution with weights proportional to $w_i$, that is, $A \sim {\rm Categ}_{(n:m)}(w_{n:m})$ if
$\P(A=k) = w_k/\big(\sum_{i=n}^m w_i\big)$ for $k \in \{n, \ldots, m \}$. 
The random variables generated in the particle filter form the following empirical approximation to the ideal predictor $\eta_k$ and filter $\pi_k$:
\[
    \eta_k^N = \frac{1}{N}\sum_{i=1}^N \delta_{X_k^i}
    \qquad\text{and}\qquad
    \pi_k^N = \Psi_k(\eta_k^N) = 
    \sum_{i=1}^N W_k^i \delta_{X_k^i},
\]
where $W_k^i = G_k(X_k^i)/\big(\sum_{j=1}^N G_k(X_k^j)\big)$.

We state two results from Del Moral \citep{del2004feynman}. The first is the exponential forgetting of the ideal filter in Lemma \ref{lem:fkcontract}, from where the second result in Lemma \ref{lem:pfstability}, on the time-uniform stability of the particle approximation (in the $L^p$), follows. These results, and our new ones to follow, assume that the FK model satisfies the following `strong mixing' condition:
\begin{assumption*}[A1]
There exist constants $0 < \underline{M} \le \bar{M} <\infty$ and
$0 < \underline{G} \le \bar{G}<\infty$ such that for all $k\ge 1$, and all $x,y \in E$:
\begin{itemize}
  \item $(M)$: The Markov transitions admit densities $M_k(x,\uarg)$ with respect to a common dominating measure $\lambda$ with 
  $\underline{M} \le M_k(x,y) \le \bar{M}$.
  \item $(G)$: The potentials satisfy $\underline{G} \le G_{k-1}(x) \le \bar{G}$. 
\end{itemize}
\end{assumption*}

\begin{lemma}(\cite{del2004feynman}, Proposition 4.3.6)
  \label{lem:fkcontract}
  Assume \textup{(A1)}. Then it holds that  \[
   \sup_{\mu, \nu} \| \Phi_{n,n+k}(\mu) - \Phi_{n,n+k}(\nu) \|_{\mathrm{TV}} \leq \beta^{k}, \quad \forall n, k\geq 0, 
   \]
where $\beta = 1-(\underline M / \bar M)^2$.
\end{lemma}
\begin{lemma}(\cite{del2004feynman}, Theorem 7.4.4) 
  \label{lem:pfstability}
  Assume \textup{(A1)} and $\osc(\phi) \leq 1$. For all $N\ge 2$, $n\ge 0$ and $p\ge 1$,
  \[
  \|\eta_{n}^N(\phi) - \eta_{n}(\phi) \|_p \leq \frac{c(p)}{\sqrt N},
  \]
  where $c(p) = 2d(p)^{1/p} (\bar M / \underline M)^3 \bar G / \underline G$. The function $d(p)$ is defined in \cite{del2004feynman}, and in particular, $d(2) = 1$.
\end{lemma}

\section{Forgetting of the particle filter}\label{sec:pfforget}

The aim of this section is to study the forgetting properties of the particle filter when its sequence of outputs, namely $\{X_t^{1:N}\}_{t\geq0}$, is viewed as an (inhomogeneous) Markov process evolving in $E^N$.  To do so, we define its corresponding Markov transition kernel:
\[
\M_n(x^{1:N}, \uarg) = \Big(\Phi_n\Big(\frac{1}{N}\sum_{i=1}^N \delta_{x^i}\Big)\Big)^{\otimes N},\quad n\geq 1.
\]
We observe that the particles produced by Algorithm~\ref{alg:pf} satisfy $X_t^{1:N} \mid \{ X_{t-1}^{1:N} = x^{1:N} \} \sim \M_t (x^{1:N}, \uarg)$ and therefore $X_t^{1:N} \sim \eta_0^{\otimes N} \M_{0,t}$
where $\M_{i,i} = {\rm Id}$ and $\M_{i,j} = \M_{i+1}\cdots \M_j$ for $i,j \geq 0$, $i < j$.

The main result of this section is Theorem \ref{thm:pfforget}. We postpone its proof to Section \ref{sec:pfforget-proof}.
\begin{theorem} 
\label{thm:pfforget}
Assume \textup{(A1)}. For all $k\geq 1,n \geq 0,N \geq 2$, 
\[
\beta_{\rm TV}(\M_{n,n+k}) \le (1 - \varepsilon)^{\lfloor k/(c \log N) \rfloor},
\]
where $\varepsilon$ and $c$ only depend on the constants in \textup{(A1)}.
\end{theorem}
A less sharp, but perhaps more convenient upper bound is found by noting that
\[
\beta_{\rm TV}(\M_{n,n+k}) \,\le\, (1 - \varepsilon)^{\lfloor k/(c \log N) \rfloor} \,\le\, (1 - \varepsilon)^{k/(c \log N)-1} \,=\, \rho r_N^k,
\]
with $\rho = (1-\varepsilon)^{-1}$ and $r_N = (1-\varepsilon)^{1/(c\log N)}$. It follows that, for any $\delta \in (0,1)$, $k = O(\log N)$ steps are sufficient to make the total variation distance smaller than $\delta$. In Section \ref{sec:example} we present an example of a model for which this rate is optimal.

We conclude this section with two remarks: the first gives (likely conservative) quantitative estimates for the constants above, and the second connects the forgetting result to empirical measures.

\begin{remark}
\label{rem:estimates}
The main point of Theorem \ref{thm:pfforget} is the $O(\log N)$ dependence of the forgetting time on the number of particles, with constants that are independent of $N$. If quantitative constants are desired, they can be extracted from the proof as follows. Let
\begin{gather*}
c \,=\, 1/\log(\beta_0^{-1}) + 2/\log 2, \quad \beta_0 = 1 - ( \underline M /\bar M )^2, \quad \varepsilon \,=\, \min\{1-\beta, 1-\alpha \}, \; \text{ where} \\
\beta \,=\, 1 - ( \underline M /\bar M )^{2\lfloor c' \rfloor}, \quad \alpha \,=\, \sqrt{ 1-(1-c'/(\lfloor c ' \rfloor +1))^{2(\lfloor c' \rfloor +1)}}, \quad c' \,=\, \frac{9}{2} (\bar M / \underline M)^8 (\bar G / \underline G)^4.
\end{gather*}
We do not expect these constants to be sharp. Indeed, the example in Section \ref{sec:example} exhibits milder dependence on the ratio $\bar M/\underline M$ (see Remark \ref{rem:deltaepsilon} for details). For $N>c'$, if the contraction constant is allowed to depend on $N$, the proof gives the sharper choice
\[
\varepsilon_N \;=\; 1-\sqrt{1-(1-c'/N)^{2N}} \;\,\xrightarrow{N \to \infty}\;\, 1-\sqrt{1-e^{-2 c'}} \;\sim\; \frac{1}{2} e^{-2c'} \quad (c' \to \infty),
\]
which decays exponentially, but behaves better (w.r.t. $c'$) than the $N$-independent value $\varepsilon$ above.
\end{remark}

\begin{remark}
The above result is formulated for the $N$-dimensional vector-valued process, but it can be viewed in terms of empirical measures as well. Indeed, in view of the identity $\| \mu - \nu \|_\mathrm{TV} = \inf_{X \sim \mu, Y \sim \nu} \P(X \neq Y)$, we can write
\[
\inf \P(\eta_k^N \neq \tilde \eta_k^N ) = \inf \P(X^{1:N}_k \neq \tilde X^{1:N}_k) \leq (1 - \varepsilon)^{\lfloor k/(c \log N) \rfloor},
\]
where $X^{1:N}_k \sim \M_{0,k}(x^{1:N}, \cdot)$ and $\tilde X^{1:N}_k \sim \M_{0,k} (\tilde x^{1:N}, \cdot)$, $(\eta_k^N, \tilde \eta_k^N)$ are the corresponding empirical measures, and the equality follows from Lemma \ref{lem:tvequality}. 
\end{remark}

\subsection{Relationship with existing forgetting result}

There are few results about forgetting in the particle filtering literature. 
  One relevant result we are aware of is in \cite[Lemma 5.1, eqn. (52)]{tadic-doucet-recursive}.
Their result, which we restate below as Lemma \ref{lem:easybound}, turns out to be very loose for large $N$. In particular, $k = O(\epsilon^{-N})$ steps, for some $\epsilon\in(0,1)$, are needed to forget the initialisation. 

The following generic upper bound similar to \citep[(4.5)]{hoeffding-wolfowitz} for the total variation distance between two product measures, expressed in terms of the total variation distances between their marginals, plays a central role in establishing Lemma \ref{lem:easybound}.

\begin{lemma}
\label{lem:producttv}
Let $\mu_{1},\ldots,\mu_{n}$ and $\nu_1, \ldots, \nu_n$ be probability measures on $E$. Let $\mu$ and $\nu$ be the product measures on $E^n$ defined by $
\mu = \mu_{1}\otimes \cdots \otimes\mu_{n}$ and $\nu = \nu_{1}\otimes \cdots \otimes\nu_{n}$.
Then
\begin{align}
\label{eq:producttv}
\left\Vert \mu-\nu\right\Vert _{\mathrm{TV}}\leq1-\prod_{i=1}^{n}(1-\left\Vert \mu_{i}-\nu_{i}\right\Vert _{\mathrm{TV}}).
\end{align}
\end{lemma}

\begin{proof}
Without loss of generality, assume that the measures $\mu_{1},\ldots,\mu_{n}, \nu_1, \ldots, \nu_n$
have a common dominating measure $\lambda$ and let $p_{i}={\rm d}\mu_{i}/{\rm d}\lambda$, $q_{i}={\rm d}\nu_{i}/{\rm d}\lambda$.
Then, 
\begin{align*}
\left\Vert \mu- \nu \right\Vert _{\mathrm{TV}}  =1-\int_{E^{n}}\min \big(p_{1}\cdots p_{n},q_{1}\cdots q_{n}\big){\rm d}\lambda^{n}
 & \leq1-\int_{E^{n}}\big(p_{1}\wedge q_{1}\big)\cdots\big(p_{n}\wedge q_{n}\big){\rm d}\lambda^{n}\\
 & =1-\int_{E}\big(p_{1}\wedge q_{1}\big){\rm d}\lambda\cdots\int_{E}\big(p_{n}\wedge q_{n}\big){\rm d}\lambda,
\end{align*}
and since $\left\Vert \mu_{i}-\nu_{i}\right\Vert _{\mathrm{TV}}=1-\int_{E}\left(p_{i}\wedge q_{i}\right){\rm d}\lambda$, we conclude.
\end{proof}

\begin{lemma}
\label{lem:easybound}
Assume \textup{(A1)}. For all $N \geq 1, n \geq 0$ and $k \geq 1$,
\[
\beta_{\rm TV}(\M_{n,n+k}) \leq (1 - \epsilon^N)^k,
\]
where $\epsilon = (\underline M / \bar M)^2$.
\end{lemma}
\begin{proof}
Let $j \geq 1$ and define $\mu_x = \Phi_{j} (\frac{1}{N}\sum_{i=1}^N \delta_{x^i} )$,
for arbitrary $x = x^{1:N} \in E^N$, so that $\mu_x^{\otimes N} = \M_{j}(x, \uarg)$.
Lemma \ref{lem:producttv} gives the bound
\begin{align*}
\beta_{\rm TV}(\M_{j}) &\;= \sup_{x, \tilde x \in E^N} \|\M_{j}(x, \uarg) - \M_{j}(\tilde x, \uarg) \|_\mathrm{TV} 
\;\leq \sup_{x, \tilde x \in E^N} \Big(1 - (1 - \| \mu_x - \mu_{\tilde x} \|_\mathrm{TV}  )^N \Big),
\end{align*}
where $\| \mu_x - \mu_{\tilde x} \|_\mathrm{TV} \leq \beta \in [0,1)$ with $\beta = 1- (\underline M / \bar M)^2$ by Lemma~\ref{lem:fkcontract}, and it follows that $\beta_{\rm TV}(\M_j) \leq 1-(1 - \beta)^N$.
Since the above argument does not depend on $j$, we conclude that
\[
\beta_{\rm TV}(\M_{n,n+k}) = \beta_{\rm TV} \bigg(\prod_{i=1}^k \M_{n+i} \bigg) \leq \,   \prod_{i=1}^k \beta_{\rm TV}(\M_{n+i}) \leq (1 - (1-\beta)^N)^k.
\qedhere
\]
\end{proof}

\subsection{Proof of Theorem \ref{thm:pfforget}}
\label{sec:pfforget-proof}

In contrast with Lemma \ref{lem:producttv}, we rely on the Hellinger distance for controlling the behaviour of product measures, and consequently, their total variation.
\begin{definition}
The squared Hellinger distance between two probability measures $P$ and $Q$ with densities $p$ and $q$ with respect to a common dominating measure $\lambda$ is
\[
H^{2}(P,Q)=\frac{1}{2}\int\big(\sqrt{p(x)}-\sqrt{q(x)}\big)^{2} \lambda({\rm d}x)=1-\int\sqrt{p(x)q(x)} \lambda({\rm d}x).
\]
\end{definition}

\begin{definition}
For a random probability measure $\mu$, we denote by $\E \mu$ its expectation, defined by
\[
(\E \mu)(A) = \E(\mu(A)),
\]
for every measurable set $A$.
\end{definition}

\begin{lemma}
  \label{lem:simpler-hellinger}
  Let $\mu$ and $\nu$ be random probability measures. Then,
  $$
  \| \E\mu - \E\nu \|_{\mathrm{TV}} 
  \le \E \| \mu - \nu \|_{\mathrm{TV}} 
  \le \sqrt{1 - \big(1-\E H^2(\mu,\nu)\big)^2}.
  $$
\end{lemma}
\begin{proof}
We may write $\| \E\mu - \E\nu \|_{\mathrm{TV}}$ as
$$
  \sup_A | \E[\mu(A) - \nu(A)] | \le 
   \sup_A \E| \mu(A) - \nu(A) | \le
  \E \sup_A | \mu(A) - \nu(A) |,
$$
which is equal to $\E \| \mu - \nu \|_{\mathrm{TV}}$. The second inequality of the claim follows from Lemma \ref{lem:tvbound},
$$
  \| \mu - \nu \|_{\mathrm{TV}}^2  \le 1 - \big(1 - H^2(\mu,\nu)\big)^2,
$$
leading to
$$
\E \| \mu - \nu \|_{\mathrm{TV}} 
  \le \E \sqrt{1 - \big(1-H^2(\mu,\nu)\big)^2},
$$
and we conclude by Jensen's inequality, with concave $t\mapsto \sqrt{t}$ and then $t\mapsto -t^2$.
\end{proof}

The following result states two bounds for product measures which we will use for two regimes, for `small $N$' and for `large $N$', respectively.
The second bound is a small extension, to random measures, of the well-known fact that one can bound total variation distance by Hellinger distance and the latter tensorises (see, e.g., \cite{wainwright2019high}, Section~15.1.3).

\begin{lemma}
\label{lem:dimfree}
Let $\mu$ and $\nu$ be random
probability measures on $E$, and let $M$ be a Markov kernel. Then, for any $N \geq 1$:
\begin{align}
  \E\big\Vert  \left(\mu M\right)^{\otimes N} -  \left(\nu M\right)^{\otimes N} \big\Vert _{\mathrm{TV}}    & \leq1-\left(1-\beta_{\rm TV}(M)\right)^{N}, \tag{a}\label{eq:dimbound-1} \\
  \E\big\Vert   \left(\mu M\right)^{\otimes N} - \left(\nu M\right)^{\otimes N} \big\Vert _{\mathrm{TV}} 
 & \leq\sqrt{1-\left(1-\E\left\{ H^{2}(\mu M, \nu M)\right\} \right)^{2N}}.
\tag{b}\label{eq:dimbound-2}
\end{align}
\end{lemma}
\begin{proof}
The first claim follows by Lemma \ref{lem:producttv}:
\begin{align*}
\E \big\Vert  \left(\mu M\right)^{\otimes N} - \left(\nu M\right)^{\otimes N} \big\Vert _{\mathrm{TV}} &\leq \E \left( 1 - (1-\|\mu M - \nu M\|_{\mathrm{TV}})^N \right),
\end{align*}
where $\|\mu M - \nu M\|_{\mathrm{TV}} \leq \beta_{\rm TV}(M) \|\mu - \nu\|_{\mathrm{TV}} \leq \beta_{\rm TV}(M)$. The second inequality of Lemma \ref{lem:simpler-hellinger} gives
\[
  \E\big\Vert  \left(\mu M\right)^{\otimes N} - \left(\nu M\right)^{\otimes N} \big\Vert _{\mathrm{TV}} \leq \sqrt{1 - \big(1-\E H^2((\mu M)^{\otimes N}, (\nu M)^{\otimes N})\big)^2},
\]
and the second claim follows directly by applying the identity $H^2(\mu_1^{\otimes N}, \mu_2^{\otimes N}) = 1- (1-H^2(\mu_1, \mu_2))^N$ and Jensen's inequality.
\end{proof}

The following result allows the transference of time-uniform stability in Lemma~\ref{lem:pfstability} from predictive approximations $\eta_n^N$ to filtering approximations $\pi_n^N$. This is essentially a generalisation of the first part of \cite[Lemma 4.3.1]{del2004feynman} for $p>1$, and is based on similar decompositions. We postpone the proof to Appendix \ref{app:technical}.

\begin{lemma}
\label{lem:bgtransformlp}
Let $\mu$ and $\nu$ be random probability measures and $p\geq 1$. If \textup{(A1)} holds, then 
\[
\sup_{\osc(\phi) \leq 1} \|\Psi_n(\mu)(\phi) - \Psi_n(\nu)(\phi) \|_p \leq c \sup_{\osc(\phi) \leq 1} \|\mu(\phi) - \nu(\phi) \|_p,
\]
where $c= \bar G / \underline G$.
\end{lemma}

The proof of the following theorem is the main argument used to prove Theorem~\ref{thm:pfforget}.
Central to its approach is the observation that the TV distance between $\M_{n,n+k}(x_0^{1:N},\uarg)$ and $\M_{n,n+k}(\tilde x_0^{1:N},\uarg)$ may be bounded by the expected TV distance between two product particle predictive measures; see \eqref{eq:tvlemnewproof-1}.
This distance can then be bounded by the corresponding expected Hellinger distance, which tensorises and therefore can be bounded precisely in terms of the expected Hellinger distance between the two particle predictive measures. In order to obtain a non-degenerate limit as $N \to \infty$, this expected distance must decrease as $O(1/N)$; see first displayed equation after \eqref{eq:tvlemnewproof-1}.
To obtain an appropriate bound we combine time-uniform particle approximation errors with forgetting of the idealised filter; see \eqref{eq:tvlemnewproof-2}. The latter gives rise to the $O(\log N)$ scaling of $k$. 

\begin{theorem}
\label{thm:tvlemmanew}
Assume \textup{(A1)}. There exist $c,c' > 1$, only depending on the constants in \textup{(A1)}, such that for all $n \geq 0$, $N \geq c'$ and $k \geq c \log(N)$,
\[
\beta_{\rm TV}(\M_{n,n+k}) \leq (1-\varepsilon^2)^{1/2},
\]
where $\varepsilon = (1-c'/N)^N$. In particular, we can choose
\[
c = \frac{1}{2 \log (\beta^{-1})} + \frac{1}{\log 2}, \quad  c' = \frac{9}{2} \left( \bar M / \underline M \right)^8 \left(\bar G / \underline G\right)^4,
\]
where $\beta = 1- (\underline M / \bar M)^2$.
\end{theorem}
\begin{proof}
Without loss of generality, assume that $n = 0$. For arbitrary $k \geq 2$, $x_0^{1:N}$ and $\tilde x_0^{1:N}$ define 
\[
\xi_t = \Phi_{0,t}\Big (\frac{1}{N} \sum_{j=1}^N \delta_{x_0^j} \Big ), \quad \tilde \xi_t = \Phi_{0,t}\Big (\frac{1}{N} \sum_{j=1}^N \delta_{\tilde x_0^j} \Big ), \quad t = 1, \ldots, k-1.
\]
Clearly, $\xi_t = \Phi_{t}(\xi_{t-1})$ for $t \in \{2,\ldots,k-1\}$.
Let $(X_{t}^{1:N})_{t=0}^{k-2}$ and $(\tilde X_{t}^{1:N})_{t=0}^{k-2}$ be the  particles generated by Algorithm \ref{alg:pf} using the parameter vectors
\[
(\xi_1, M_{2:k-1}, G_{1:k-2}, N) \quad \text{and} \quad (\tilde \xi_1, M_{2:k-1}, G_{1:k-2}, N).
\] 
Let $(\xi_t^N$, $\tilde \xi_t^N)_{t=1}^{k-1} = (N^{-1} \sum_{j=1}^N \delta_{X_{t-1}^j}, N^{-1} \sum_{j=1}^N \delta_{\tilde X_{t-1}^j})_{t=1}^{k-1}$ be the corresponding empirical measures.
Note that $(\xi_t^N, \tilde \xi_t^N)$ are the particle approximations of the measures $(\xi_t, \tilde \xi_t)$.
We also have $X_{t-1}^{1:N} \sim \M_{0,t}(x_0^{1:N}, \uarg)$ and $\tilde X_{t-1}^{1:N} \sim \M_{0,t}(\tilde x_0^{1:N}, \uarg)$, for $t=1,\ldots,k-1$.

A final important relationship we will exploit is as follows. Define $\mu = \Psi_{k-1}(\xi_{k-1}^N)$, $\nu = \Psi_{k-1}(\tilde \xi_{k-1}^N)$, and $M = M_k$. Then, $\mu M = \Phi_{k}(\xi_{k-1}^N)$,  $\nu M = \Phi_{k}(\tilde \xi_{k-1}^N)$, and
\begin{align*}
\E \big\{ \left(\mu M\right)^{\otimes N} \big\} = \M_{0,k}(x_0^{1:N}, \uarg), \quad 
\E \big\{ \left(\nu M\right)^{\otimes N} \big\} = \M_{0,k}(\tilde x_0^{1:N}, \uarg).
\end{align*}

The first part of the proof is to relate the total variation distance between $\M_{0,k}(x_0^{1:N}, \uarg)$ and $\M_{0,k}(\tilde x_0^{1:N}, \uarg)$ to the expected squared Hellinger distance between two product particle predictive measures, and then to the expected squared Hellinger distance between the particle predictive measures themselves. This is done via bound \eqref{eq:dimbound-2} of Lemma~\ref{lem:dimfree} together with Lemma \ref{lem:simpler-hellinger}:
\begin{align}
\big\| \M_{0,k}(x_0^{1:N}, \uarg) - \M_{0,k}(\tilde x_0^{1:N}, \uarg) \big\|_{\mathrm{TV}} 
&\leq \E \,\big\Vert   \left(\mu M \right)^{\otimes N} - \left(\nu M \right)^{\otimes N} \big\Vert _{\mathrm{TV}} \nonumber \\
&\leq \sqrt{1-\left(1-\E H^{2}(\mu M, \nu M) \right)^{2N}} \nonumber  \\
&= \sqrt{1-\left(1-\E H^{2}(\Phi_{k}(\xi_{k-1}^N), \Phi_{k}(\tilde \xi_{k-1}^N)) \right)^{2N}}. \label{eq:tvlemnewproof-1}
\end{align}
We now show that the expected squared Hellinger distance between $\Phi_{k}(\xi_{k-1}^N)$ and $\Phi_{k}(\tilde \xi_{k-1}^N)$ is of order $1/N$, which will `stabilise' the exponent $2N$ above. In particular, the claim will follow once we show
\[
 \E \big[ H^2(\Phi_{k}(\xi_{k-1}^N), \Phi_{k}(\tilde \xi_{k-1}^N)) \big] \leq c'/N, 
\]
for all $N \geq c'$, and arbitrary $x_0^{1:N}$ and $\tilde x_0^{1:N}$. Denote the squared Hellinger distance above by
\[
H^2_{k} := H^2(\Phi_{k}(\xi_{k-1}^N), \Phi_{k}(\tilde \xi_{k-1}^N)).
\]
Applying Lemma~\ref{lem:h2bound} to $\Phi_{k}(\xi_{k-1}^N) = \Psi_{k-1}(\xi_{k-1}^N)M_{k}$ and $\Phi_{k}(\tilde \xi_{k-1}^N) = \Psi_{k-1}(\tilde \xi_{k-1}^N)M_{k}$  yields
\[
\E H^2_{k} \leq C \sup_{\osc(\phi) \leq 1} \E (|\Psi_{k-1}(\xi_{k-1}^N)(\phi) - \Psi_{k-1}(\tilde \xi_{k-1}^N)(\phi)|^2),
\]
where $C = (1/8) (\bar M/ \underline M)^2$. 
We apply Lemma \ref{lem:bgtransformlp} (with $p=2$) to obtain
\begin{align*}
\sup_{\osc(\phi) \leq 1} \| \Psi_{k-1}(\xi_{k-1}^N)(\phi) - \Psi_{k-1}(\tilde \xi_{k-1}^N)(\phi) \|_2 &\leq 
C_2 \sup_{\osc(\phi) \leq 1} \| \xi_{k-1}^N(\phi) - \tilde \xi_{k-1}^N(\phi) \|_2,
\end{align*}
with $C_2 = \bar G/ \underline G$. Next, we bound the $L^2$-norm on the right-hand side. Minkowski's inequality gives
\begin{align}
\|\xi_{k-1}^N(\phi) - \tilde{\xi}_{k-1}^N(\phi) \|_2 \leq \, & \|\xi_{k-1}^N(\phi) - \xi_{k-1}(\phi) \|_2  + |\xi_{k-1}(\phi) - \tilde \xi_{k-1}(\phi) | \, + \nonumber \\ 
& \|  \tilde \xi_{k-1}^N(\phi) - \tilde \xi_{k-1}(\phi) \|_2,
\label{eq:tvlemnewproof-2}
\end{align}
where we note that the second term on the right-hand side is nonrandom. Lemma~\ref{lem:pfstability} (time-uniform particle filter stability) gives the bounds
\begin{align*}
\|\xi_{k-1}^N(\phi) - \xi_{k-1}(\phi) \|_2 \leq \frac{C_3}{\sqrt N}, \qquad 
\|\tilde \xi_{k-1}^N(\phi) - \tilde \xi_{k-1}(\phi) \|_2 \leq \frac{C_3}{\sqrt N}, 
\end{align*}
where $C_3=2(\bar M / \underline M)^3 \bar G / \underline G$, and by Lemma~\ref{lem:fkcontract} (exponential forgetting of the ideal filter),
\[
|\xi_{k-1}(\phi) - \tilde \xi_{k-1}(\phi) | \leq \sup_{\mu, \nu} \| \Phi_{0,k-1}(\mu) - \Phi_{0,k-1}(\nu) \|_{\mathrm{TV}} \leq \beta^{k-1}
\]
for some $\beta \in [0,1)$. If $ k \geq c \log(N)$ with 
\[
c = \frac{1}{2 \log (\beta^{-1})} + \frac{1}{\log 2},
\]
then $k -1 \geq \log(N)/(2\log(\beta^{-1}))$ for all $N\geq 2$. It follows that $\beta^{k-1} \leq C_3/\sqrt{N}$ and
\[
\sup_{\osc(\phi) \leq 1} \|\xi_{k-1}^N(\phi) - \tilde{\xi}_{k-1}^N(\phi) \|_2 \leq 3\frac{C_3}{\sqrt N}.
\] 
Combining these bounds, we conclude that $\E H^2_{k} \leq c' N^{-1}$ with $c'=9C C_2^2 C_3^2 = (9/2) (\bar M / \underline M)^8 (\bar G / \underline G)^4$. \qedhere
\end{proof}

  \begin{remark}\unskip
    \label{rem:random-particle-filter-one-step-bound}
    The proof of Theorem \ref{thm:tvlemmanew} is based on the intermediate bound (see \eqref{eq:tvlemnewproof-1}) of the form
\begin{equation*}
\sup_{x_n^{1:N}, \, \tilde{x}_n^{1:N}} \E \left\| \M_{n+k}(X_{n+k-1}^{1:N}, \uarg) - \M_{n+k}(\tilde{X}_{n+k-1}^{1:N},\uarg)\right\|_{\mathrm{TV}} 
\leq (1-\varepsilon^2)^{1/2},
\end{equation*}
where $X_{n+k-1}^{1:N}\sim \M_{n,n+k-1}(x_{n}^{1:N},\uarg)$ and $\tilde{X}_{n+k-1}^{1:N}\sim \M_{n,n+k-1}(\tilde{x}_{n}^{1:N}, \uarg)$, that is, particle filters started from configurations $x_{n}^{1:N}$ and $\tilde{x}_{n}^{1:N}$, respectively, at time $n$ and propagated to time $n+k-1$. This bound holds for any $N\ge c'$ and $k\ge c\log N$, and only depends on the marginal distributions of the filters; this has consequences for the analysis of coupled particle filters, which we discuss further in Section \ref{sec:couplings}.
  \end{remark}

We are now in a position to provide a suitable simplification of Theorem~\ref{thm:tvlemmanew} and prove Theorem~\ref{thm:pfforget}.

\begin{lemma}
\label{lem:tvuniform}
Assume \textup{(A1)}. There exist $N_{\rm min} \geq 2$, $c>1$,  and $\alpha \in [0,1)$, only depending on the constants in \textup{(A1)}, such that for all $n\geq 0$, $N \geq N_{\rm min}$ and $k \geq c \log(N)$,
\[
\beta_{\rm TV} \left( \M_{n,n+k} \right) \leq  \alpha.
\]
\end{lemma}

\begin{proof}
Let $c$ and $c'$ be as in Theorem~\ref{thm:tvlemmanew}. By Lemma~\ref{lem:derivative} we observe that $N \mapsto (1-(1-c'/N)^{2N})^{1/2}$ is decreasing for $N > c'$.
Hence, we conclude by using Theorem~\ref{thm:tvlemmanew} with $N_{\rm min} = \lfloor c' \rfloor +1$ and $\alpha = (1-(1-c'/N_{\rm min})^{2N_{\rm min}})^{1/2}$.
\end{proof}

\begin{proof}[Proof of Theorem \ref{thm:pfforget}]
Let $N_{\rm min}$ be as in Lemma \ref{lem:tvuniform}. For $N < N_{\rm min}$, Lemma \ref{lem:easybound} gives
\[
\beta_{\rm TV} \Big( \prod_{i=1}^k \M_{n+i} \Big) \leq (1 - \epsilon^N)^k ,
\]
where $\epsilon = (\underline M / \bar M)^2$.
Since
\begin{align*}
(1 - \epsilon^N)^k \leq (1 - \epsilon^{N_{\rm min}-1})^k &\leq (1 - \epsilon^{N_{\rm min}-1})^{ (k \log 2)/(\log N)} \leq (1 - \epsilon^{N_{\rm min}-1})^{ \lfloor (k \log 2)/(\log N) \rfloor},
\end{align*}
the result holds with $c=1/\log 2$ and $\varepsilon = (\underline M / \bar M)^{2(N_{\rm min}-1)}$.

For $N \geq N_{\rm min}$, we divide the time steps $n+1, \ldots n+k$ into blocks of size $b$, $0 < b \leq k$,
and we obtain that there exist constants $\alpha \in [0,1)$ and $C>1$ such that for all $b \geq C \log N$,
\[
\beta_{\rm TV}(\M_{n,n+k}) \leq \prod_{j=0}^{\lfloor k/b \rfloor-1} \beta_{\rm TV} (\M_{n+jb,n+(j+1)b}) \leq \alpha^{\lfloor k/b \rfloor},
\]
where we have used the inequality $\beta_{\rm TV}(\M_i) \leq 1$ for  $i>n+\lfloor k/b \rfloor b$ and Lemma \ref{lem:tvuniform}.
We will now verify that the claim holds with the constants $\varepsilon = 1-\alpha$, $c = 2C$. Fix $b = \lceil C \log N \rceil$ and assume that $k \geq 2C \log N$ (otherwise the result is trivial).
The inequality $\lceil x \rceil \leq 2x$, $\forall x\geq 1/2$ implies 
\[
\alpha^{\lfloor k/b \rfloor} = \alpha^{\lfloor k / \lceil C \log N \rceil \rfloor } \leq \alpha^{\lfloor k / (2 C \log N)\rfloor }.
\]
The proof is now completed by combining the cases $N < N_{\rm min}$ and $N \geq N_{\rm min}$ above, i.e., by choosing the larger constant $c$ and the smaller constant $\varepsilon$.
\end{proof}

\begin{remark}
  The proof of Theorem~\ref{thm:tvlemmanew} involves bounding the expected TV distance
  between product distributions using the Hellinger distance and tensorisation.
  It is informative to consider whether the upper bound of Lemma~\ref{lem:producttv} could be used instead,
  while retaining the property that the analysis relies only on time-uniform
  stability of the particle approximations and exponential forgetting
  of the ideal filter. To that end, we have (with $\pi_{k-1}^N = \Psi_{k-1}(\xi_{k-1}^N)$, $\tilde \pi_{k-1}^N = \Psi_{k-1}(\tilde \xi_{k-1}^N)$):
\begin{align}
   \mathbb{E}\big\| \big(\pi_{k-1}^{N}M_{k}\big)^{\otimes N}\!\!\!-\big(\tilde{\pi}_{k-1}^{N}M_{k}\big)^{\otimes N}\big\Vert _{{\rm TV}} 
    \leq1-\big(1-\mathbb{E}\big\Vert \pi_{k-1}^{N}M_{k}-\tilde{\pi}_{k-1}^{N}M_{k}\big\Vert _{{\rm TV}}\big)^{N},
   \label{eq:pure-TV-bound-product}
\end{align}
  by Lemma~\ref{lem:producttv} and Jensen's inequality. Denoting $\pi_{k-1} = \Psi_{k-1}(\xi_{k-1})$ and $\tilde \pi_{k-1} = \Psi_{k-1}(\tilde \xi_{k-1})$, a natural counterpart to \eqref{eq:tvlemnewproof-2}
  is:
  \begin{align*}
  \mathbb{E}\left\Vert \pi_{k-1}^{N}M_{k}-\tilde{\pi}_{k-1}^{N}M_{k}\right\Vert _{{\rm TV}} \leq \; & \mathbb{E}\left\Vert \pi_{k-1}^{N}M_{k}-\pi_{k-1}M_{k}\right\Vert _{{\rm TV}}+\left\Vert \pi_{k-1}M_{k}-\tilde{\pi}_{k-1}M_{k}\right\Vert _{{\rm TV}}+\\
   & \mathbb{E}\left\Vert \tilde{\pi}_{k-1}^{N}M_{k}-\tilde{\pi}_{k-1}M_{k}\right\Vert _{{\rm TV}}.
  \end{align*}
  By Lemma \ref{lem:TV-rootN-upper} in Appendix \ref{app:technical}, the first and third terms on the right can be upper bounded by $O(N^{-1/2})$. Ignoring the second term (which vanishes exponentially in $k$) yields an upper bound in \eqref{eq:pure-TV-bound-product} that degrades to $1$ as $N\to\infty$. The order $O(N^{-1/2})$ of the upper bound above is also expected to be tight, as demonstrated by Lemma \ref{lem:independent-tv} in a simplified context, where $\pi_{k-1}^N$ and $\tilde{\pi}_{k-1}^N$ consist of independent realisations from $\pi_{k-1}$ and $\tilde{\pi}_{k-1}$, respectively.
\end{remark}

\section{An example where $\log(N)$ steps are necessary and sufficient}
\label{sec:example}
We study a model where $\Omega(\log N)$ steps are required for the forgetting of the particle filter.
More precisely, we will see that $\lim_{N \to \infty}\beta_{\rm TV}(\M_{0,k(N)}) = 1$ when $k(N) \leq c \log N$ for sufficiently small $c>0$. 
This allows us to conclude that the $k=O(\log N)$ dependence in Theorem \ref{thm:pfforget} cannot be improved in general.

\begin{example}
  \label{exm:badmodel}
Let $E=\{0,1\}$, fix $\varepsilon \in (0, 1/2)$ and define 
$G_{n}\equiv 1$ and
\begin{equation*}
M_{n}(x,\uarg)=(1-\varepsilon)\delta_{x} + \varepsilon\delta_{1-x}, 
\end{equation*}
for all $n \geq 1$. One may verify that
\[
\Phi_n(\mu) = \Big((1-\varepsilon) \mu(0) + \varepsilon \mu(1)\Big) \delta_0 + \Big( \varepsilon \mu(0) + (1-\varepsilon)\mu(1)\Big) \delta_1,
\]
and consequently
\[
\M_{n}(x^{1:N}, \uarg) = \bigg( \frac{1}{N}  \sum_{i=1}^N \Big(1-\varepsilon-(1-2\varepsilon)x^i\Big)  \delta_0 + \frac{1}{N}  \sum_{i=1}^N \Big(\varepsilon + (1-2\varepsilon)x^i \Big) \delta_1  \bigg)^{\otimes N}.
\]
\end{example}
This model satisfies \textup{(A1)}, and Theorem \ref{thm:pfforget} gives that, for all initial states $x_0^{1:N}$ and $\tilde x_0^{1:N}$,
\[
\limsup_{N \to \infty} \| \M_{0,k(N)}(x_0^{1:N}, \uarg) - \M_{0,k(N)}(\tilde x_0^{1:N}, \uarg) \|_{\mathrm{TV}} < 1
\]
if $k(N) \geq c \log N$ for sufficiently large $c>0$. The converse is given by the following theorem. 
\begin{theorem}
\label{thm:simpler-example}
Consider Example \ref{exm:badmodel} with some $\varepsilon \in (0, 1/2)$. Let $x_0^{1:N} = (1, \ldots, 1)$ and $\tilde x_0^{1:N} = (0, \ldots, 0)$. There exists a constant $\delta_\varepsilon>0$, only depending on $\varepsilon$, such that
\[
\lim_{N\rightarrow \infty}\| \M_{0,k(N)}(x_0^{1:N}, \uarg) - \M_{0,k(N)}(\tilde x_0^{1:N}, \uarg) \|_{\mathrm{TV}} = 1
\]
for any sequence $\left( k(N)\right) _{N=1}^\infty$ in $\mathbb{N}$ such that $ \delta_\varepsilon \log N - k(N) \to \infty$. In particular, we can choose  $\delta_\varepsilon = 1/\log((1-2\varepsilon)^{-2})$.
\end{theorem}
We postpone the proof to Section \ref{sec:qmarginals}, where we further study the properties of this particle filter. 

\begin{remark}
\label{rem:deltaepsilon}
For this example, with counting measure as dominating measure, we have $\underline M = \varepsilon$, $\bar M = 1-\varepsilon$, $\underline G = \bar G = 1$,
and therefore
\[
\delta_\varepsilon \;\sim\; \frac{1}{4} \frac{1 - \varepsilon}{\varepsilon} \;=\; \frac{1}{4} \frac{\bar M}{\underline M} \quad (\varepsilon \to 0).
\]
We observe that the coefficient of the logarithmic time scale increases only linearly in $\bar M/ \underline M$, which is milder than the dependence obtained from the general bounds in Section \ref{sec:pfforget}. Since  $\bar G/\underline G = 1$, the filtering operator $\Phi_n$ reduces to a Markov transition and, in this example, admits a one-step contraction in total variation. Since such a one-step contraction is not available in general under (A1), it is perhaps not surprising to see some improvement in the constants. 
\end{remark}
\section{Forgetting of the conditional particle filter}
\label{sec:cpfforget}
In this section, we study the forgetting properties of the conditional particle filter  (CPF) \cite{andrieu-doucet-holenstein}, whose execution is given in Algorithm \ref{alg:cpf} below. Unlike the particle filter, it has an added input, which is the `frozen' or `reference' path $x^*_{0:T-1}\in E^T$ placed at $X^0_{0:T-1}$ in the algorithm. The reference particles participate in the selection step; for instance, if $A_{k}^{i}=0$, then $X_k^0$ is pushed through the Markov transition $M_{k+1}$ to define $X_{k+1}^i$. Note that if the reference is regarded as one particle, which is common in the literature, then Algorithm \ref{alg:cpf} has $N+1$ particles.

\begin{algorithm}
  \caption{\textsc{ConditionalParticleFilter}$(\eta_0, M_{1:T}, G_{0:T-1}, x^*_{0:T-1},N)$}
  \label{alg:cpf} 
\begin{algorithmic}[1]
\State
Draw $X_0^{\ind{i}} \sim \eta_0(\uarg)$, \quad $i = 1 \ldots N$
\For{$k=0,\ldots,T-1$}
\State Set $X_k^0 = x_k^*$ \label{item:cpf-reference}
\State Draw $A_{k}^{\ind{1:N}} \,\stackrel{i.i.d.}{\sim} \,\mathrm{Categ}_{(0:N)}\big(
G_{k}(X_{k}^{\ind{0}}),\ \dots, G_{k}(X_{k}^{\ind{N}})\big)$ \label{item:cpf-selection}
\State Draw $X_{k+1}^{\ind{i}} \sim M_{k+1}(X_{k}^{\ind{A_{k}^{\ind{i}}}}, \uarg)$, \quad $i \in \{1, \ldots, N\}$
\EndFor
\end{algorithmic}
\end{algorithm}

The conditional particle filter was introduced in \cite{andrieu-doucet-holenstein} where the authors also showed how it corresponds precisely to a Gibbs sampler (in the context of Markov Chain Monte Carlo) for a particular probability distribution. This equivalence gave rise to the appellation `conditional', and `particle filter' for its algorithmic similarity to Algorithm \ref{alg:pf}. We further refer the reader to \cite{andrieu-doucet-holenstein} for the use of the conditional particle filter in Bayesian estimation for time-series data.

The CPF described in Algorithm \ref{alg:cpf} can be viewed as a `perturbed' particle filter. The perturbation is in the selection step due to the presence of the reference particle; see lines \ref{item:cpf-reference} and \ref{item:cpf-selection} of Algorithm \ref{alg:cpf}.
We denote, analogous to the particle filter, by $\hat \Psi_k$ and $\hat \Phi_k$ the perturbed selection and perturbed (combined) selection-mutation operators:
\begin{align}
\hat \Psi_k(\mu) &:= \Psi_k \Big(\frac{1}{N+1} \delta_{x_k^*} + \frac{N}{N+1}\mu \Big), \qquad \hat \Phi_k (\mu) := \Phi_k \Big(\frac{1}{N+1} \delta_{x_{k-1}^*} + \frac{N}{N+1}\mu \Big). \label{eq:perturbselmut}
\end{align}
These perturbed operators suggest that the output of the CPF approximates the true filter similarly to the particle filter. We remark that this viewpoint of the conditional particle filter, which is important in our proof to establish its forgetting properties, is atypical in the literature  \cite{andrieu-doucet-holenstein}.

We will investigate the behaviour of the particles generated in Algorithm \ref{alg:cpf} --- that is, excluding the reference --- as a Markov chain $(X^{1:N}_k)_{k=0}^{T}$. They follow the Markov transitions given by
\[
\M_k^{x_{k-1}^*}(x^{1:N}, \uarg) 
\,=\, \Big( \Phi_k \Big( \frac{1}{N+1}   \delta_{x_{k-1}^*} \!\!+ \frac{1}{N+1}\sum_{i=1}^N \delta_{x^i} \Big) \Big)^{\otimes N}
=\, \Big( \hat\Phi_k\Big( \frac{1}{N} \sum_{i=1}^N \delta_{x^i}\Big)\Big)^{\otimes N},
\]
for arbitrary $x_{k-1}^* \in E$ and $x^{1:N} \in E^N$. We use the shorthand notations $x^* = x_{0:T-1}^*$ and $\M_{i,j}^{ x^*} = \M_{i+1}^{x^*_{i}} \cdots \M_{j}^{x^*_{j-1}}$.
The empirical predictive and filtering distributions generated by the CPF are defined by
\[
\hat \eta_k^N := \frac{1}{N}\sum_{i=1}^N \delta_{X_k^i}
\qquad\text{and}\qquad
\hat\pi_k^N := \hat \Psi_k(\hat\eta_k^N),
\]
respectively, where $X_k^{1:N}$ are the particles in Algorithm \ref{alg:cpf} at time $k$. 

Our main forgetting result for the CPF is analogous to Theorem \ref{thm:pfforget}. We postpone the proof details and auxiliary results to Appendix \ref{app:technical}.
\begin{theorem} 
\label{thm:cpfforget}
Assume \textup{(A1)}. For all $k \geq 1, n\geq 0, N \geq 2,$ and references $x^*,$
\[
\beta_\mathrm{TV}(\M_{n,n+k}^{ x^*} ) \leq (1-\varepsilon)^{\lfloor k/(c \log N)\rfloor},
\]
where $\varepsilon$ and $c$ only depend on the constants in \textup{(A1)}.
\end{theorem}
\begin{proof}
We argue exactly as in the proof of Theorem \ref{thm:pfforget}, but replace Lemmas \ref{lem:easybound} and \ref{lem:tvuniform} by Lemmas \ref{lem:easybound_cpf} and \ref{lem:tvuniform_cpf}, respectively.
  \end{proof}

\section{$L^p$ bounds for the conditional particle filter}
\label{sec:cpfstability}

In this section, we prove time-uniform $L^p$ error estimates between the conditional particle filter and the true filter. The result is necessary for the proof of the CPF forgetting result (Theorem \ref{thm:cpfforget}) above, but can be of independent interest for understanding the behaviour of the CPF with long time horizons.

The main result of this  section is the following time-uniform error, which is analogous to Lemma \ref{lem:pfstability}, but stated for the CPF and filtering distributions:
\begin{theorem}
\label{thm:cpfstability}
Assume \textup{(A1)}. For every $p \geq 1$, there exists a constant $c(p)$ such that for all $\phi$ with $\osc(\phi) \leq 1$, all $n\geq 0$, $N\ge 1$ and all references $x^*$,
\begin{align*}
\big\|\hat \pi_n^N(\phi) - \pi_n(\phi)\big\|_p \leq  &\, \frac{c(p)}{\sqrt{N}}.
\end{align*}
\end{theorem}

Theorem \ref{thm:cpfstability} complements the total variation error estimates of \cite{kuhlenschmidt-singh} between the conditional particle filter and the true filter. In the proof, more precisely in Eq. \eqref{eq:cpfconstant}, we give an explicit value for $c(p)$.

The remainder of this section is dedicated to the proof of Theorem \ref{thm:cpfstability}. We introduce the mappings $Q_{k,n}$ from $E$ to measures on $E$, and the Markov kernels $P_{k,n}$
\begin{gather*}
Q_{k,n} (\phi )(x_k) := \int \lambda({\rm d}x_{k+1}) \ldots \lambda({\rm d}x_n) \phi(x_n) \prod_{i=k}^{n-1} G_i (x_i)M_{i+1}(x_i,x_{i+1} )\\
P_{k,n}(\phi)(x_k) := \frac{Q_{k,n}(G_n \phi)(x_k) }{Q_{k,n}(G_n)(x_k)},
\end{gather*}
and observe that with this notation,
\[
\Phi_{k,n}(\mu)(\phi) = \frac{\mu Q_{k,n}(\phi)}{\mu Q_{k,n}(1)}.
\]
The mappings $Q_{k,n}$ have the following stability property:
\begin{lemma}
\label{lem:Qbounds}
Assume \textup{(A1)}. Then there exists $c\in [1, \infty)$ such that for all $k \geq 0,n\geq k$ and $x, x'$
\[
0 < Q_{k,n}(G_n)(x) \leq c Q_{k,n}(G_n)(x').
\]
In particular, one may take $c = (\bar G/\underline G) (\bar M/\underline M)$.
\end{lemma}

\begin{proof}
Assume $n>k$; the case $n=k$ is trivial. Since $G_k, \ldots, G_n$ and $M_{k+1}, \ldots, M_n$ are bounded from above and below, it holds that $0 < Q_{k,n}(G_n) < \infty$.
Moreover, $M_{k+1}(x,y)/M_{k+1}(x',y) \leq \bar M/\underline M$. Hence,
\begin{align*}
  Q_{k,n}(G_n)(x) \,&=\, G_k(x) \int \lambda({\rm d}x_{k+1}) Q_{k+1,n}(G_n)(x_{k+1}) M_{k+1}(x, x_{k+1}) \\
    &\leq\, \frac{\bar G}{\underline G}  \frac{\bar M}{\underline M} G_k(x') \int \lambda({\rm d}x_{k+1}) Q_{k+1,n}(G_n)(x_{k+1}) M_{k+1}(x', x_{k+1}) \,=\, \frac{\bar G}{\underline G} \frac{\bar M}{\underline M} Q_{k,n}(G_n)(x'),
\end{align*}
from which we may conclude.
\end{proof}

The following is a restatement of \citep[Lemma A.1]{del2015uniform}, but with explicit constants.
\begin{lemma}
\label{lem:martingaleineq}
 Let $G'$ and $h'$ be measurable functions $\mathcal{X} \to \mathbb{R}$ with $0 < G'(x) \leq c'G'(x') $, $\forall x, x'$ and some $c' \in [1, \infty)$. Let $X_1, \ldots, X_N$ be independent samples from $\mu$. If $\osc(h') < \infty$, then for any $p \geq 1$
\[
\sqrt{N} \, \bigg\| \frac{\sum_{i=1}^N G'(X_i)h'(X_i)}{\sum_{i=1}^N G'(X_i)} - \frac{\mu(G'h')}{\mu(G')}\bigg\|_p \,\leq\, d(p)^{1/p} c' \operatorname{osc}(h'),
\]
where $d(p)$ is as in Lemma \ref{lem:pfstability}.
\end{lemma}

The first direct result towards proving Theorem \ref{thm:cpfstability} is Theorem \ref{thm:lpbound_cpf} stated next. To derive it, we follow a similar decomposition of errors used in \cite{kuhlenschmidt-singh} to establish total variation error. In particular, the sum of the errors on the right-hand side of Theorem \ref{thm:lpbound_cpf}  is due to the (one-step) Monte Carlo errors and the (one-step) errors due to the perturbation of the selection operators. These are given, respectively, by the terms $T_k$ in \eqref{eq:T_k} and  $R_k$ in \eqref{eq:R_k}.

\begin{theorem}
\label{thm:lpbound_cpf}
Assume \textup{(A1)}, and let $c$ be such that $0 < Q_{k,n}(G_n)(x) \leq c Q_{k,n}(G_n)(x')$ for all $k=0, \ldots, n$. Then, for all $N \geq 1$, $p \geq 1$, $x^*$ and $\phi: E \to \mathbb{R}$ with $\osc(\phi) <\infty$, 
\begin{align*}
\left\|\hat \Psi_n (\hat \eta_n^N)(\phi) - \Psi_n(\eta_n)(\phi) \right\|_p \leq  \, \|Z_n \|_p \osc(\phi) + \frac{c \,d(p)^{1/p} }{ \sqrt{N}}\sum_{k=0}^n q_{k,n}
 + \sum_{k=0}^{n-1} \| \gamma_{k,n}\|_p \, q_{k,n},
\end{align*}
where
\begin{gather*}
q_{k,n} = \osc(P_{k,n}(\phi)), \quad Z_n = \frac{G_n( x_n^*)}{N\hat\eta_n^N (G_n) + G_n( x_n^*)} , \quad 
\gamma_{k,n} = \frac{\delta_{x_k^*} Q_{k,n}(G_n)}{(N \hat \eta_{k}^N+\delta_{x_k^*}) Q_{k,n}(G_n)}.
\end{gather*}
\end{theorem}
\begin{proof}
Recall that $\pi_n = \Psi_n(\eta_n)$. By Minkowski's inequality,
\begin{align*}
\label{eq:firstdecomposition}
\| \hat \Psi_n(\hat \eta_n^N)(\phi) - \pi_n(\phi) \|_p \leq \| \hat \Psi_n (\hat \eta_n^N )(\phi) - \Psi_n (\hat \eta_n^N)(\phi) \|_p + \| \Psi_n (\hat \eta_n^N)(\phi) - \pi_n(\phi) \|_p. 
\end{align*}
By Lemma \ref{lem:hatpsibound}, the first term on the right-hand side is bounded by $\|Z_n\|_p \osc(\phi)$. For the second term, we use the identities
\[
\Psi_n (\hat \eta_n^N)(\phi) = \frac{\hat \eta_n^N (G_n \phi)}{\hat \eta_n^N (G_n)} \quad \text{and} \quad \pi_n(\phi) = \frac{\eta_0 Q_{0,n}(G_n \phi)}{\eta_0 Q_{0,n}(G_n)},
\]
which together with Minkowski's inequality and a telescopic decomposition give the bound
\begin{align*}
\left\| \Psi_n (\hat \eta_n^N)(\phi) - \pi_n(\phi) \right\|_p \leq \left\|T_0 \right\|_p + \sum_{k=1}^n \left( \| T_k \|_p + \| R_k \|_p \right),
\end{align*}
where
\begin{gather}
T_0 := \frac{\hat \eta_0^N Q_{0,n}(G_n \phi)}{\hat \eta_0^N Q_{0,n}(G_n)} - \frac{\eta_0 Q_{0,n}(G_n \phi)}{\eta_0 Q_{0,n}(G_n)}, \quad 
T_k := \frac{\hat \eta_k^N Q_{k,n}(G_n \phi)}{\hat \eta_k^N Q_{k,n}(G_n)} - \frac{(\hat \Phi_k \hat \eta_{k-1}^N) Q_{k,n}(G_n \phi)}{(\hat \Phi_k \hat \eta_{k-1}^N) Q_{k,n}(G_n)}, \, k \geq 1, \label{eq:T_k} \\
R_k := \frac{(\hat \Phi_k \hat \eta_{k-1}^N) Q_{k,n}(G_n \phi)}{(\hat \Phi_k \hat \eta_{k-1}^N) Q_{k,n}(G_n)} - \frac{\hat \eta_{k-1}^N Q_{k-1,n}(G_n \phi)}{\hat \eta_{k-1}^N Q_{k-1,n}(G_n)}. \label{eq:R_k}
\end{gather}
By Lemma \ref{lem:rknbound}, $\| R_k \|_p \leq \left\| \gamma_{k-1,n} \right\|_p \osc(P_{k-1,n} (\phi))$, and it remains to bound $T_i$, $i=0,\ldots, n$. 

Denote $\mathcal{F}_k=\sigma\{ X_{0:k}^{1:N}, A_{0:k-1}^{1:N}\}$ and write $\E[|T_k|^p] = \E[\E(|T_k|^p\,|\, \mathcal{F}_{k-1})]$ for $k\geq 1$. Conditionally on $\mathcal{F}_{k-1}$, the random variables $X_k^1, \ldots,  X_k^N$ are i.i.d. samples from $\hat \Phi_k \hat \eta_{k-1}^N$. 
We observe that
\[
T_k = \frac{\mu^N(G'h')}{\mu^N(G')} - \frac{\mu(G'h')}{\mu(G')},
\]
where 
\begin{gather*}
G' = Q_{k,n}(G_n), \quad h' = P_{k,n}(\phi), \quad \mu^N = \hat\eta_k^N, \quad  \mu = \hat \Phi_k \hat \eta_{k-1}^N,
\end{gather*}
and consequently $G' h' = Q_{k,n}(G_n \phi)$.
Then, by Lemma \ref{lem:martingaleineq},
\[
\E(|T_k|^p\,|\, \mathcal{F}_{k-1}) \leq \big( N^{-1/2} d(p)^{1/p} c \operatorname{osc}(h') \big)^p,
\]
and so $\E[|T_k|^p]^{1/p} \leq N^{-1/2} d(p)^{1/p} c \operatorname{osc}(h')$. The case $T_0$ follows by the same argument with $G' = Q_{0,n}(G_n)$, $h' = P_{0,n}(\phi)$, $\mu^N = \hat \eta_0^N$ and $\mu = \eta_0$.
\end{proof}

\begin{lemma}
\label{lem:hatpsibound}
Let $\phi: E \to \mathbb{R}$ be a function with $\osc(\phi) < \infty$. Then
\[
\| \hat \Psi_k(\hat \eta_k^N) (\phi) - \Psi_k(\hat \eta_k^N) (\phi)\|_p \leq \| Z_k \|_p \osc(\phi),
\]
where
\[
Z_k = \frac{G_k(x_k^*)}{ N \hat \eta_k^N(G_k) + G_k( x_k^*)}.
\]
\end{lemma}
\begin{proof}
For any probability measure $\mu$ on $E$ it holds that
\begin{align*}
 \hat \Psi_k(\mu)(\phi) \,&:=\,  \frac{N\mu (G_k \phi) + G_k(x_k^*)\phi(x_k^*)}{N \mu(G_k) + G_k(x_k^*)} \,=\, \frac{\mu(G_k \phi)}{\mu(G_k)}\Big(\frac{N \mu(G_k)}{N \mu(G_k) + G_k(x_k^*)} \Big) \,+\, \frac{G_k(x_k^*)\phi(x_k^*)}{ N \mu(G_k) + G_k(x_k^*)} \\
&=\, \Psi_k(\mu)(\phi) \Big(1- \frac{G_k(x_k^*)}{N \mu(G_k) + G_k(x_k^*)} \Big) + \frac{G_k(x_k^*)}{ N \mu(G_k) + G_k(x_k^*)} \phi(x_k^*),
\end{align*}
hence
\begin{align*}
\hat \Psi_k(\mu)(\phi) - \Psi_k(\mu)(\phi) &=  (\phi(x_k^*) - \Psi_k(\mu)(\phi)) \frac{G_k(x_k^*)}{ N \mu(G_k) + G_k(x_k^*)},
\end{align*}
where $|\phi(x_k^*) - \Psi_k(\mu)(\phi)| \leq \osc(\phi)$. Setting $\mu = \hat \eta_k^N$ and taking $\|\cdot\|_p$ now gives the claim.
\end{proof}

We postpone the proof of the following technical lemma to Appendix \ref{app:technical}.
\begin{lemma}
\label{lem:rknbound}
Assume \textup{(A1)}. For $R_k$ defined in \eqref{eq:R_k}, it holds that
$\|R_k\|_p \leq \|\gamma_{k-1,n}\|_p \operatorname{osc}\left( P_{k-1,n}(\phi) \right)$.
\end{lemma}

\begin{proof}[Proof of Theorem \ref{thm:cpfstability}]
We estimate each term in Theorem \ref{thm:lpbound_cpf} separately:
\begin{align*}
\| Z_n \|_p = \bigg\| \frac{G_n(x_n^*)}{N\hat\eta_n^N (G_n) \!+\! G_n(x_n^*)} \bigg\|_p \leq \frac{1}{N\!+\!1} \frac{\bar G}{\underline G}, \quad
\|\gamma_{k,n}\|_p = \bigg\| \frac{\delta_{x_k^*} Q_{k,n}(G_n)}{(N\hat \eta_{k}^N+\delta_{x_k^*}) Q_{k,n}(G_n)} \bigg\|_p \leq \frac{1}{N\!+\!1} \frac{ \bar G}{\underline G} \frac{ \bar M}{\underline M},
\end{align*}
where the last inequality follows from the fact that, by Lemma \ref{lem:Qbounds}, 
\[
\delta_{x_k^*} Q_{k,n}(G_n) \;\leq\; \frac{ \bar G}{\underline G} \frac{ \bar M}{\underline M} \hat \eta_{k}^N Q_{k,n}(G_n).
\] 
Next, we bound $\osc(P_{k,n}(\phi))$. Since $P_{k,n}(\phi)(x) = \Psi_n(\Phi_{k,n}(\delta_x))(\phi)$, it follows that
\begin{align*}
\osc(P_{k,n}(\phi)) &\;\leq\; \sup_{x,y} \sup_{\osc(\phi) \leq 1} | \Psi_n(\Phi_{k,n}(\delta_x))(\phi) - \Psi_n(\Phi_{k,n}(\delta_y))(\phi) | \\
&\;\leq\; \sup_{x,y} \| \Psi_n(\Phi_{k,n}(\delta_x)) - \Psi_n(\Phi_{k,n}(\delta_y)) \|_{\mathrm{TV}}
\end{align*}
where we have used $\| \mu - \nu \|_\mathrm{TV} = \sup_{\osc(\phi) \leq 1} |\mu(\phi) - \nu(\phi)|$.  Lemmas \ref{lem:bgtransformlp} and \ref{lem:fkcontract} then give the bound
\begin{align*}
\sup_{x,y} \| \Psi_n(\Phi_{k,n}(\delta_x)) - \Psi_n(\Phi_{k,n}(\delta_y)) \|_{\mathrm{TV}} &\;\leq\; \frac{\bar G}{\underline G} \sup_{x,y} \| \Phi_{k,n}(\delta_x) - \Phi_{k,n}(\delta_y) \|_{\mathrm{TV}} \;\leq\; \frac{\bar G}{\underline G} \beta^{n-k},
\end{align*}
hence $\sum_{k=0}^n \osc(P_{k,n}(\phi)) \leq (\bar G/ \underline G) (1-\beta)^{-1}$ for all $n$, with $\beta = 1-(\underline M / \bar M)^2$. It follows from Theorem \ref{thm:lpbound_cpf} and the previous bounds that
\begin{align*}
\left\|\hat \Psi_n (\hat \eta_n^N)(\phi) - \Psi_n(\eta_n)(\phi) \right\|_p \;\leq\;  &\, \frac{1}{N+1}\frac{\bar G }{\underline G} + \frac{\,d(p)^{1/p} }{ \sqrt{N}} \left( \frac{\bar G}{\underline G} \right)^2 \left( \frac{\bar M}{\underline M} \right)^3  + \frac{1}{N+1} \left( \frac{\bar G}{\underline G} \right)^2 \left( \frac{\bar M}{\underline M} \right)^3.
\end{align*}
Since $1/(N+1) \leq 1/(2\sqrt{N})$ for $N\geq 1$, the claim follows with the constant
\begin{equation}
\label{eq:cpfconstant}
c(p) = \frac{\bar G}{ 2 \underline G} + \Big( \frac{\bar G}{\underline G} \frac{\bar M}{\underline M} \Big)^2 \frac{\bar M}{\underline M} \Big( d(p)^{1/p} + \frac{1}{2} \Big), 
\end{equation}
where $d(p)$ is as in Lemma \ref{lem:pfstability}.
\qedhere
\end{proof}

From Theorem \ref{thm:lpbound_cpf} we almost immediately obtain a CPF analogue of Lemma \ref{lem:pfstability}:
\begin{theorem} 
\label{thm:cpfstability_predictive}
Assume \textup{(A1)} and $\osc(\phi) \leq 1$. For every $p \geq 1$ and for all $N\geq 1$, $n\geq 0$, and $x^*$,
\[
\|\hat \eta_{n}^N(\phi) - \eta_{n}(\phi) \|_p \leq \frac{c(p)}{\sqrt N},
\]
where $c(p)$ only depends on $p$ and the constants in \textup{(A1)}.
\end{theorem}
\begin{proof}
By the proof of Theorem \ref{thm:lpbound_cpf},
\[
 \| \Psi_n \hat \eta_n^N(\phi) - \Psi_n \eta_n(\phi) \|_p \leq \frac{c' d(p)^{1/p} }{ \sqrt{N}}\sum_{k=0}^n  \osc(P_{k,n} (\phi)) + \sum_{k=0}^{n-1} \| \gamma_{k,n}\|_p \osc(P_{k,n}(\phi))
\]
for some $c'$. We set $G_n = 1$, so that $\| \hat \eta_n^N(\phi) - \eta_n(\phi)\|_p = \| \Psi_n \hat \eta_n^N(\phi) - \Psi_n \eta_n(\phi) \|_p$.
The claim now follows by bounding $\| \gamma_{k,n}\|_p$ and $\osc(P_{k,n} (\phi))$ as in Theorem \ref{thm:cpfstability}.
\end{proof}

\section{Connections to propagation of chaos}\label{sec:pochaos}

Using essentially the same proof technique as in our main result (Theorem \ref{thm:pfforget}), we can derive new
`propagation of chaos' type bounds for particle filters. The aim of such results is to quantify how close the joint law of $q$ particles is to the law of independent draws from the corresponding ideal filter. As opposed to the  forgetting results, here the initial states are not fixed and arbitrary, but generated from the true initial distribution $\eta_0$. 

\begin{theorem}
  \label{thm:uniform-poc}
  Assume \textup{(A1)}. Let $X_k^i$ be the random variables generated in the particle filter (Algorithm \ref{alg:pf}), and let $\eta_k$ denote the corresponding ideal filter. There exists a constant $C\in(0,\infty)$ which depends only on the constants in \textup{(A1)}, such that for all $k\ge 1$, $N > C$ and $1 \le q\le N$ the following holds:
  \begin{equation}
   \| \mathrm{Law}(X_k^{1:q}) -  \eta_k^{\otimes q} \|_\mathrm{TV}
   \le \sqrt{1 - \Big(1 - \frac{C}{N}\Big)^{2q} } \le  \sqrt{2C\frac{q}{N}}.
   \label{eq:poc-bound-hellinger}
  \end{equation}
\end{theorem}
\begin{proof}
We may write, similar to the proof of Theorem \ref{thm:tvlemmanew}, 
\begin{align*}
  \| \mathrm{Law}(X_k^{1:q}) -  \eta_k^{\otimes q} \|_\mathrm{TV}^2
& \;=\; \| \E \big\{( \pi_{k-1}^N M_k )^{\otimes q}\} -  \eta_k^{\otimes q} \|_\mathrm{TV}^2  \;\le\; 1 - \big(1 - \E H^2 (\pi_{k-1}^N M_k, \eta_k) \big)^{2q}.
\end{align*}
Lemmas \ref{lem:h2bound} and \ref{lem:bgtransformlp} yield, with $C_1 = \frac{1}{8}(\bar{M}/\underline{M})^2(\bar{G}/\underline{G})^2$:
\begin{align*}
\E H^2 (\pi_{k-1}^N M_k, \eta_k) 
&= \E H^2 \big(\Psi_{k-1}(\eta_{k-1}^N) M_k, \Psi_{k-1}(\eta_{k-1}) M_k\big) \le  C_1\sup_{\mathrm{osc}(\phi)\le 1} \| \eta_{k-1}^N(\phi) - \eta_{k-1}(\phi) \|_2^2, 
\end{align*}
and Lemma \ref{lem:pfstability} ensures that $\| \eta_{k-1}^N(\phi) - \eta_{k-1}(\phi) \|_2^2 \le c^2/N$ where $c = 2(\bar{M}/\underline{M})^3\bar{G}/\underline{G}$. Therefore, the claimed upper bounds hold with $C = c^2C_1$:
$$
  \| \mathrm{Law}(X_k^{1:q}) -  \eta_k^{\otimes q} \|_\mathrm{TV}^2 \le 1 - \Big(1 - \frac{C}{N}\Big)^{2q} 
  \le 2 C \frac{q}{N},
$$
where the latter form follows by Bernoulli's inequality.
\end{proof}

The upper bound in Theorem \ref{thm:uniform-poc} has different dependence on $q$ and $N$ than other results in the literature. Namely, \citep{delmoral-doucet-peters} established an upper bound of the following form (under milder assumptions):
\begin{equation}
  \| \mathrm{Law}(X_k^{1:q}) -  \eta_k^{\otimes q} \|_\mathrm{TV} \le c_k \frac{q^2}{N}.
  \label{eq:poc-bound-literature}
\end{equation}
With $q$ fixed and $N\to\infty$, the $O(N^{-1})$ dependence in \eqref{eq:poc-bound-literature} is superior to the $O(N^{-1/2})$ rate of \eqref{eq:poc-bound-hellinger}. However, for large $q$, the bounds in \eqref{eq:poc-bound-hellinger} can remain useful, unlike \eqref{eq:poc-bound-literature}. Even with $q=N$, the first upper bound in \eqref{eq:poc-bound-hellinger} has a nontrivial limit $\sqrt{1 - e^{-2C}}$ as $N\to\infty$, and with $q = o(N)$ the bound tends to zero.

We computed (numerically) the exact values for the `propagation of chaos' total variations in \eqref{eq:poc-bound-hellinger} for a discrete state-space model to investigate how they behave with different combinations of $q$ and $N$. The model was the same one as in Section \ref{sec:example}, with $\epsilon=0.1$, but having a non-uniform potential: $G_k(0) = 0.1$ and $G_k(1)=1$. Figure \ref{fig:poc} shows the computed (logarithmic) total variations as a function of (logarithmic) ratios $q/N$, for two time instances: $k=4$ and $k=20$. To illustrate the different orders of decay, we added lines to the figures of orders $\Theta(q/N)$ and $\Theta(\sqrt{q/N})$. The lines are chosen to be minimal bounds that upper bound all total variations for $N\ge 64$ and $1 \le q \le N$. We also added a $\Theta(q/N)$ line which lower bounds all the displayed distances.

\begin{figure}
  \includegraphics[width=\linewidth]{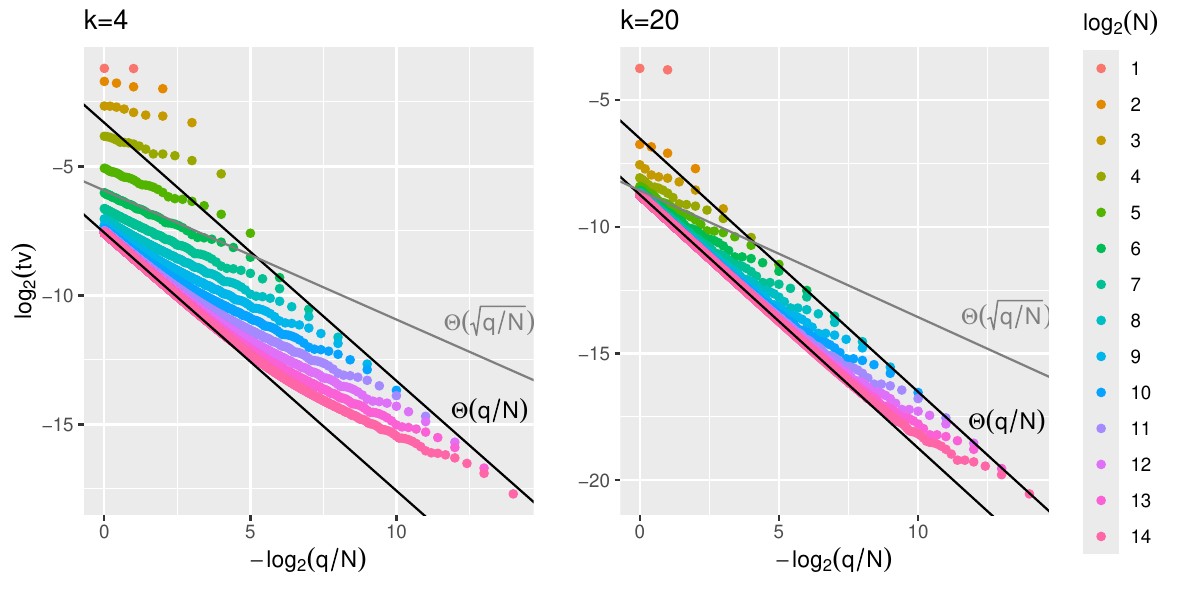}
  \caption{Propagation of chaos total variation distances in \eqref{eq:poc-bound-hellinger}  as a function of $q/N$ in a discrete state-space model example for times $k=4$ and $k=20$. The lines illustrate the mentioned orders of decay.}
  \label{fig:poc}
\end{figure}

For fixed $q$ and increasing $N$, the total variations follow the $O(1/N)$ asymptotic as \eqref{eq:poc-bound-literature} suggests. With $q=N$ (i.e.~$\log_2(q/N)=0$), the total variations seem to converge to a non-trivial value as $N\to\infty$, as our bound \eqref{eq:poc-bound-hellinger} suggests. Curiously, with fixed $N$ and varying $q$, some of the patterns seem to exhibit roughly $O(\sqrt{q})$ behaviour for a range of $q$, and the total variations can be upper and lower bounded between $\Theta(q/N)$ curves, but these may be specific to the example. 

The propagation of chaos can be combined with forgetting, to get a version of Theorem \ref{thm:uniform-poc} for any initialisation, leading to bounds that suggest $q = o(N)$ particles at time $j+k$ are nearly i.i.d.~from the ideal filter and independent of $X_{j}^{1:N}$, as long as $N$ is large enough and $k\ge c\log N$. 
We denote below by $\Pi(\uarg)_q$ the projection to the first $q$ marginals, that is, for $X^{1:N} \sim \mu$ we have $X^{1:q} \sim \Pi(\mu)_q$.

\begin{theorem}
  \label{thm:forget-poc}
   Assume \textup{(A1)}. There exist $c>0$ and $C>0$, only depending on the constants in \textup{(A1)}, such that for all $N > C$, $k \geq c \log N$, $1 \leq q \leq N$, and all initial states $x_0^{1:N}$ it holds that
\begin{equation*}
 \| \Pi\big(\M_{0,k}(x_{0}^{1:N}, \uarg)\big)_q -  \eta_k^{\otimes q} \|_\mathrm{TV}
 \le \sqrt{1 - \Big(1 - \frac{C}{N}\Big)^{2q} } \le  \sqrt{2C\frac{q}{N}}.
\end{equation*}
\end{theorem}
\begin{proof}
  Using the same notation $(\xi_{t}^N$, $\xi_t)$ as in Theorem \ref{thm:tvlemmanew}, we argue as in the proof of Theorem \ref{thm:uniform-poc}:
\begin{align*}
 \| \Pi\big(\M_{0,k}(x_{0}^{1:N}, \uarg)\big)_q -  \eta_k^{\otimes q} \|_\mathrm{TV}^2
&\le 1 - \big(1 - \E H^2 (\Phi_k (\xi_{k-1}^N), \Phi_k(\eta_{k-1})) \big)^{2q} \\
&\le 1 - \big(1 - C_1 \sup_{\mathrm{osc}(\phi)\le 1} \| \xi_{k-1}^N(\phi) - \eta_{k-1}(\phi) \|_2^2 \big)^{2q}.
\end{align*}
Using Lemma \ref{lem:pfstability}, we further bound $\| \xi_{k-1}^N(\phi) - \eta_{k-1}(\phi) \|_2$ by
\begin{align*}
\| \xi_{k-1}^N(\phi) - \xi_{k-1}(\phi)  \|_2
+ \| \xi_{k-1}(\phi) - \eta_{k-1}(\phi)  \|_2  
 \;\le\; \frac{C_2}{\sqrt{N}} + \| \xi_{k-1}(\phi) - \eta_{k-1}(\phi)  \|_2,
\end{align*}
and by Lemma \ref{lem:fkcontract}, we conclude that this is at most $2C_2/\sqrt N$ when $k$ satisfies $\beta^{k-1} \leq C_2/\sqrt{N}$ (e.g., when $k \geq c \log N$ with $c = 1/( 2\log(\beta^{-1})) + 1/\log2$).
The claim follows with $C = 4 C_1 C_2^2$.
\end{proof}

Analogues of Theorems \ref{thm:uniform-poc} and \ref{thm:forget-poc} also hold for the conditional particle filter. We omit the proofs, which are almost identical to above, except they also employ our $L^p$ bound for the CPF (Theorem \ref{thm:cpfstability}) in place of Lemmas \ref{lem:bgtransformlp} and \ref{lem:pfstability}. We recall the shorthand notations $x^* = x_{0:k-1}^*$ and $\M_{0,k}^{ x^*} = \M_{1}^{x^*_{0}} \cdots \M_{k}^{x^*_{k-1}}$.
\begin{theorem}
  \label{thm:uniform-poc-cpf}
  Assume \textup{(A1)}. Let $X_k^i$ be as in Algorithm \ref{alg:cpf}, and let $\eta_k$ denote the corresponding ideal filter. There exists $C\in(0,\infty)$ which depends only on the constants in \textup{(A1)}, such that for all $k\ge 1$, $N > C$, $1 \le q\le N$, and references $x^*$ it holds that:
  \begin{equation*}
   \| \mathrm{Law}(X_k^{1:q}) -  \eta_k^{\otimes q} \|_\mathrm{TV}
   \le \sqrt{1 - \Big(1 - \frac{C}{N}\Big)^{2q} } \le  \sqrt{2C\frac{q}{N}}.
  \end{equation*}
\end{theorem}
\begin{theorem}
  \label{thm:forget-poc-cpf}
    Assume \textup{(A1)}. There exist $c>0$ and $C>0$, only depending on the constants in \textup{(A1)}, such that for all $N > C$, $k \geq c \log N$, $1 \leq q \leq N$, and all initial states $x_0^{1:N}$ and references $x^*$ it holds that
\begin{equation*}
 \| \Pi\big(\M^{x^*}_{0,k}(x_{0}^{1:N}, \uarg)\big)_q -  \eta_k^{\otimes q} \|_\mathrm{TV}
 \le \sqrt{1 - \Big(1 - \frac{C}{N}\Big)^{2q} } \le  \sqrt{2C\frac{q}{N}}.
\end{equation*}
\end{theorem}
\section{Implications for algorithms that couple particle filters}
\label{sec:couplings}

Couplings of particle filters have been suggested in the context of multilevel Monte Carlo \citep[e.g.][]{gregory-cotter-reich,jasra-kamatani-law-zhou,sen-thiery-jasra,jasra-yu}.  Couplings of CPFs have been used in their theoretical analysis \citep{chopin-singh,lee-singh-vihola} and in algorithms that generate unbiased estimates for smoothing expectations \citep{jacob-lindsten-schon,lee-singh-vihola}. We review here some implications of our results in the coupling context.

In what follows, $\textsc{MaxCouple}(\mu,\nu)$ stands for a total variation maximal coupling  \cite[cf.][]{thorissoncoupling}; that is, if $(X,Y) \sim \textsc{MaxCouple}(\mu,\nu)$, then $X\sim \mu$, $Y\sim \nu$ and $\mathbb{P}(X\neq Y) = \| \mu - \nu \|_{\mathrm{TV}}$. In practice, we can simulate from a maximal coupling if we can sample from $\mu$ and $\nu$ and pointwise evaluate their densities with respect to a common dominating measure. Namely, if we draw $X\sim \mu$ and use Algorithm \ref{alg:maxcouple-generic} to simulate $Y\gets \textsc{CondMaxCouple}\big(X, \mu, \nu)$, it is straightforward to verify that $(X,Y) \sim \textsc{MaxCouple}(\mu,\nu)$.
\begin{algorithm}[H]
  \caption{\textsc{CondMaxCouple}$\big(X, \mu, \nu)$}
  \label{alg:maxcouple-generic}
  \begin{algorithmic}[1]
  \State \textbf{with probability} $1\wedge \frac{\nu(X)}{\mu(X)}$
  \textbf{output} $X$
  \Loop
  \State Draw $Y \sim \nu$
  \State \textbf{with probability} $1 - \big(1 \wedge \frac{\mu(Y)}{\nu(Y)}\big)$
  \textbf{output} $Y$
  \EndLoop
  \end{algorithmic}
\end{algorithm}

Our results are directly relevant when two particle filters share common dynamics (models), except for having different initialisations. In such a context, let $\vec{X}_0\neq \tilde{\vec{X}}_0$ be two arbitrary initial particle filter states, and denote the predictive distributions of the filters as
$\mu_{k} := \Phi_k(N^{-1} \sum_{i=1}^N \delta_{X_{k-1}^i} )$ and 
$\tilde{\mu}_{k} := \Phi_k(N^{-1} \sum_{i=1}^N \delta_{\tilde{X}_{k-1}^i} )$.
One step of the coupling algorithm suggested in \citep[Section 3.3]{jasra-yu} may be summarised as follows:
\begin{equation}
  (X_k^i,\tilde{X}_k^i) \sim \textsc{MaxCouple}(\mu_{k} , \tilde{\mu}_{k}), \quad\text{independently for }i=1,\ldots,N.
  \label{eq:individual-maximal-coupling}
\end{equation}
We have $\P(X_k^i \neq \tilde{X}_k^i) = \E \| \mu_{k} - \tilde{\mu}_{k} \|_{\mathrm{TV}}$, which can be upper bounded using the triangle inequality:
\begin{equation*}
  \E \| \mu_{k} - \tilde{\mu}_{k} \|_{\mathrm{TV}} \,\leq\, \E \big[ \| \mu_{k} - \Phi_{1,k}(\mu_1) \|_{\mathrm{TV}} + \| \Phi_{1,k}(\mu_1) - \Phi_{1,k}(\tilde{\mu}_1)\|_{\mathrm{TV}} + \| \tilde{\mu}_{k} - \Phi_{1,k}(\tilde{\mu}_1)\|_{\mathrm{TV}} \big].
\end{equation*}
Lemma \ref{lem:TV-rootN-upper} ensures that the first and the third terms are of order $O(N^{-1/2})$, and so is the middle term, provided that $k \ge O(\log N)$, by Lemma \ref{lem:fkcontract}. Consequently, by independence and Jensen's inequality,
\begin{align}
\P(\mathbf{X}_{k} = \tilde{\mathbf{X}}_{k}) 
&\,=\, \E \big[\P(\vec{X}_k = \tilde{\vec{X}}_k \mid \vec{X}_{k-1},\tilde{\vec{X}}_{k-1}) \big] \,=\, \E \big(1 - \| \mu_{k} - \tilde{\mu}_{k} \|_\mathrm{TV}\big)^N
\,\ge\, \big(1- O(N^{-1/2})\big)^{N}.
\label{eq:individual-coupling-bound}
\end{align}

Instead of \eqref{eq:individual-maximal-coupling}, we can also maximally couple 
$\M_k(\vec{X}_{k-1},\uarg)$ and $\M_k( \tilde{\vec{X}}_{k-1}, \uarg)$, that is,
the two particle filter states:
\begin{equation}
(\vec{X}_k,\tilde{\vec{X}}_k) \sim \textsc{MaxCouple}(\mu_{k}^{\otimes N}, \tilde{\mu}_{k}^{\otimes N}).
\label{eq:state-maximal-coupling}
\end{equation}
For this algorithm, under (A1) we have that for $k \ge  c\log N$ and $N$ large enough, Remark \ref{rem:random-particle-filter-one-step-bound} implies
\begin{align}
\P(\vec{X}_k \neq \tilde{\vec{X}}_k) 
&\,=\, \E \big[\P(\vec{X}_k \neq \tilde{\vec{X}}_k \mid \vec{X}_{k-1},\tilde{\vec{X}}_{k-1}) \big] 
 \,=\, \E \| \mu_{k}^{\otimes N} - \tilde{\mu}_{k}^{\otimes N} \|_{\mathrm{TV}}
\,\le\, (1-\varepsilon^2)^{1/2}. 
\label{eq:state-maximal-coupling-bound}
\end{align} 
The bound above guarantees that $\P(\vec{X}_k = \tilde{\vec{X}}_k)$ remains bounded away from zero for large $N$, unlike \eqref{eq:individual-coupling-bound}, which degrades to zero as $N\to\infty$.
Furthermore, similar to Lemma 10 in the supplementary material to \cite{karjalainen-lee-singh-vihola-mixing}, an application of \eqref{eq:state-maximal-coupling-bound} every $\lceil c \log N \rceil$ steps yields 
$$
   \P(\vec{X}_t \neq \tilde{\vec{X}}_t, \vec{X}_{t+1} \neq \tilde{\vec{X}}_{t+1},\ldots, \vec{X}_{t+k} \neq \tilde{\vec{X}}_{t+k}) 
   \le \alpha^{\lfloor \frac{k}{\lceil c\log N \rceil } \rfloor }, \quad\text{where}\quad \alpha = (1-\varepsilon^2)^{1/2},
$$
which guarantees a coupling in $O(\log N)$ time.

Using maximal coupling for particle filter states in this setting is novel to our knowledge, and can be implemented in practice. Indeed, simulation from \eqref{eq:state-maximal-coupling} can be implemented in (expected) $O(N^2)$ time like \eqref{eq:individual-maximal-coupling}; see the discussion on similar conditional particle filter coupling algorithms in \citep{karjalainen-lee-singh-vihola-mixing}. One may view our contribution as demonstrating that \eqref{eq:state-maximal-coupling} can be reasonable when (A1) holds and possibly more generally. In fact, the conditional coupling probability $\P(\vec{X}_k = \tilde{\vec{X}}_k \mid \vec{X}_{k-1},\tilde{\vec{X}}_{k-1})$ of the filter state maximal coupling in \eqref{eq:state-maximal-coupling} is guaranteed to be superior to that of \eqref{eq:individual-maximal-coupling}. 
Moreover, setting $\mu = \mu_{k}^{\otimes N}$ and $\nu =  \tilde{\mu}_{k}^{\otimes N}$ in Lemma \ref{lem:producttv}, the term on the left and the upper bound in \eqref{eq:producttv} coincide with $\P(\vec{X}_k \neq \tilde{\vec{X}}_k \mid \vec{X}_{k-1},\tilde{\vec{X}}_{k-1})$ of \eqref{eq:state-maximal-coupling} and \eqref{eq:individual-maximal-coupling}, respectively.

We stress that the discussion above is only about conditional probabilities, and they are  upper bounds which might not be tight when the algorithms are iterated $k$ steps. It would be interesting to investigate whether \eqref{eq:individual-maximal-coupling} or \eqref{eq:state-maximal-coupling} is genuinely better in relevant scenarios. Moreover, the multilevel context of \cite{jasra-yu} differs from what is discussed above, because there the particle filter dynamics are (slightly) different.

\section{On the forgetting of the $q$-marginals}
\label{sec:qmarginals}

In Section \ref{sec:pochaos}, we considered propagation of chaos bounds for the $q$-marginals $\Pi \big( \M_{0,k}(x_0^{1:N}, \uarg) \big)_{q}$, that is, their total variation distance to the product distribution $\eta_k^{\otimes q}$. Those bounds do not converge to zero as $k$ increases.
We now return to Example \ref{exm:badmodel}, and combine ideas from the previous sections to prove a forgetting result for the $q$-marginals, which gives Theorem \ref{thm:simpler-example} as a special case. 
In this example the forgetting is guaranteed by an \enquote{$N$-independent} growth condition $\lim_N \delta_\varepsilon \log q(N) - k(N) = -\infty$, and conversely, the forgetting may fail (depending on the initial state) if the limit equals $\infty$. 

A crucial part in the proof is that when we apply the independent coupling \eqref{eq:individual-maximal-coupling} to this model, the expected Hamming distance between $X_k^{1:N}$ and $\tilde X_k^{1:N}$ is contracted at each time step $k$ by a constant factor. This fact relies on the choice of potentials $G=1$, and does not hold under (A1) in general. Secondly, the states $X_k^1, \ldots, X_k^N$ are Bernoulli instead of complicated mixtures of distributions, which simplifies the analysis.

\begin{theorem}
\label{thm:simpler-example2}
Consider Example \ref{exm:badmodel} with some $\varepsilon \in (0, 1/2)$. Let $x_0^{1:N} = (1, \ldots, 1)$ and $\tilde x_0^{1:N} = (0, \ldots, 0)$, and for each $N \geq 1$, let $q(N) \in \{1, \ldots, N \}$, $k(N) \geq 1$. With $\delta_\varepsilon = 1/\log((1-2\varepsilon)^{-2})$, it holds that
\begin{enumerate}
\renewcommand{\labelenumi}{(\roman{enumi})} 
\item If $ \delta_\varepsilon \log q(N) - k(N) \to \infty$, then
\[
\lim_{N\rightarrow \infty}\| \Pi \big( \M_{0,k(N)}(x_0^{1:N}, \uarg) \big)_{q(N)} - \Pi\big( \M_{0,k(N)}(\tilde x_0^{1:N}, \uarg) \big)_{q(N)} \|_{\mathrm{TV}} = 1.
\]
\item If $ \delta_\varepsilon \log q(N) - k(N) \to -\infty$, then
\[
\lim_{N\rightarrow \infty}\| \Pi \big( \M_{0,k(N)}(x_0^{1:N}, \uarg) \big)_{q(N)} - \Pi\big( \M_{0,k(N)}(\tilde x_0^{1:N}, \uarg) \big)_{q(N)} \|_{\mathrm{TV}} = 0.
\]
\end{enumerate}
\end{theorem}

In the proof we will use the fact that for this model, the coupling probability in the independent maximal coupling \eqref{eq:individual-maximal-coupling} can be derived exactly:
\[
\mu_k = \Phi_k \Big(N^{-1} \sum_{i=1}^N \delta_{X_{k-1}^i} \Big) \;=\; \mathrm{Ber}\Big(N^{-1} \sum_{i=1}^N X_{k-1}^i \Big) M_k,
\]
hence
\[
\mu_k(\{0\}) = 1-\mu_k(\{1\}) \;=\; \varepsilon \Big( N^{-1} \sum_{i=1}^N X_{k-1}^i \Big) + (1-\varepsilon) \Big( 1- N^{-1} \sum_{i=1}^N X_{k-1}^i \Big),
\]
and denoting $D_{k-1} = | \sum_{i=1}^N X_{k-1}^i - \sum_{i=1}^N \tilde X_{k-1}^i|$, it follows from $\|\mathrm{Ber}(p) - \mathrm{Ber}(q)\|_{\mathrm{TV}} = |p-q|$ that if $(X, \tilde X) \sim \textsc{MaxCouple}(\mu_{k} , \tilde{\mu}_{k})$, then $\|\mathrm{Law}(X) - \mathrm{Law}(\tilde X)\|_\mathrm{TV} = |\mu_k(\{1\}) - \tilde \mu_k(\{1\})|$ equals
\begin{align}
\Big|\, \varepsilon N^{-1} \sum_{i=1}^N (\tilde X_{k-1}^i - X_{k-1}^i)  + (1-\varepsilon)   N^{-1} \sum_{i=1}^N ( X_{k-1}^i - \tilde X_{k-1}^i ) \, \Big| \;=\; (1-2 \varepsilon) N^{-1} D_{k-1}. \label{eq:maxcouplebtv}
\end{align} 
One may verify that a maximal coupling here is monotone: if $\mu_k(\{1\}) \geq \tilde \mu_k(\{1\})$, then $\P (X \geq \tilde X) = 1$.

\begin{proof}
(i): Let $k \geq 1$ and $X_k^{1:N} \sim \M_{0,k}(x_0^{1:N}, \uarg)$, $Y_k^{1:N} \sim \M_{0,k}(\tilde x_0^{1:N}, \uarg)$. Define the proportions of 1's,
\begin{align}
\label{eq:proportions}
P_{k}^{q}= \frac{1}{q}\sum_{i=1}^q X_k^i, \quad \text{and} \quad \tilde P_{k}^{q}= \frac{1}{q} \sum_{i=1}^q Y_k^i,
\end{align}
and denote the distributions of $q P_{k}^{q}$ and $q \tilde P_{k}^{q}$ by $p_k$ and $\tilde p_k$, respectively.
We first establish that
\begin{align}
\label{eq:pqtv}
\| p_k - \tilde p_k \|_\mathrm{TV} \leq \| \Pi \big( \M_{0,k}(x_0^{1:N}, \uarg) \big)_{q} - \Pi \big( \M_{0,k}(\tilde x_0^{1:N}, \uarg) \big)_{q} \|_{\mathrm{TV}}.
\end{align}
Let $(\tilde X^{1:q}, \tilde Y^{1:q})$ be any coupling of $(X^{1:q}_k, Y^{1:q}_k)$. It follows that $(\sum_{i=1}^q \tilde X^i, \sum_{i=1}^q \tilde Y^i)$ forms a coupling of $(qP_k^q, q \tilde P_k^q)$. Since $\{\tilde X^{1:q} = \tilde Y^{1:q} \} \subset \{\sum_{i=1}^q \tilde X^i = \sum_{i=1}^q \tilde Y^i\}$, the coupling inequality gives
\begin{align*}
\| p_k - \tilde p_k \|_\mathrm{TV} \leq \P\Big(\sum_{i=1}^q \tilde X^i \neq \sum_{i=1}^q \tilde Y^i\Big) \leq \P\Big(\tilde X^{1:q} \neq \tilde Y^{1:q}\Big),
\end{align*}
and \eqref{eq:pqtv} follows by taking the infimum over all couplings $(\tilde X^{1:q}, \tilde Y^{1:q})$. 

We proceed by bounding $\| p_k - \tilde p_k \|_\mathrm{TV}$. Let $(P_k, \tilde P_k)$ be a maximal coupling of $(P_k^q, \tilde P_k^q)$. Now,
\begin{align*}
\| p_k - \tilde p_k \|_\mathrm{TV} 
\,=\, \P(P_k \neq \tilde{P}_k) \,\ge\, \P\Big(P_k > \frac{1}{2}, \tilde{P}_k < \frac{1}{2}\Big)
& \,=\, 1 - \P\Big(\{P_k \le \frac{1}{2}\} \cup \{ \tilde{P}_k \ge \frac{1}{2}\}\Big) \\
& \,\ge\, 1 - \P\Big(P_k \le \frac{1}{2} \Big) - \P\Big( \tilde{P}_k \ge \frac{1}{2}\Big),
\end{align*}
where by Markov's inequality
\begin{align*}
\P\Big(P_k \le \frac{1}{2}\Big)
&\,\leq\, \P\Big( \big| \E[P_k] - P_k \big| \ge \E[P_k] - \frac{1}{2}\Big) 
\,\leq\, \frac{\Var(P_k)}{(\E[P_k] - \frac{1}{2})^2}.
\end{align*}
By symmetry, the same upper bound applies to $\P(\tilde P_k \geq 1/2)$, and we use Lemma \ref{lem:Pnmoments} to bound
\[
\Var(P_k) \leq \frac{1}{4q}\frac{1}{1-(1-2\varepsilon)^2},
\]
and so
\begin{align*}
\| p_k - \tilde p_k \|_\mathrm{TV} \,\geq\, 1 - 2 \frac{\Var(P_k)}{(\E[P_k] - \frac{1}{2})^2} \,\geq\,  
1 - 2 \frac{1}{q} \frac{1}{1 - (1-2\varepsilon)^2} \frac{1}{(1-2 \varepsilon)^{2k}}.
\end{align*}
The right-hand side approaches one when $\delta_\varepsilon \log q - k \to \infty$.

(ii):
Denote again $X_k^{1:N} \sim \M_{0,k}(x_0^{1:N}, \uarg)$ and $Y_k^{1:N} \sim \M_{0,k}(\tilde x_0^{1:N}, \uarg)$. We use the independent maximal coupling in \eqref{eq:individual-maximal-coupling} for every time step $k \geq 1$, and denote $D_k := | \sum_{i=1}^N X_k^i -  \sum_{i=1}^N Y_k^i|$ and $a := 1-2\varepsilon$. 
As noted above, the maximal coupling is necessarily monotone. Since $x_0^{i} \geq \tilde x_0^{i}$, $i = 1, \ldots, N$, it follows that $X_k^{i} \geq Y_k^{i}$, $i = 1, \ldots, N$, and so $D_k = \sum_{i=1}^N \I(X_k^i \neq Y_k^i)$. By \eqref{eq:maxcouplebtv} and conditional independence of $(X_k^i, Y_k^i)_{i=1}^N$, it holds that
$D_k \mid D_{k-1} \sim \mathrm{Bin}(N, aD_{k-1}/N )$, hence $\E(D_k \mid D_{k-1}) = aD_{k-1}$ and
\begin{align*}
\E(D_k^2 \mid D_{k-1}) &\;=\; N[aD_{k-1}/N (1-aD_{k-1}/N)] + N^2(aD_{k-1}/N)^2 \;=\; aD_{k-1} + a^2D_{k-1}^2 (1-1/N),
\end{align*}
so that $\E(D_k) = N a^k$ by $D_0 = N$, and the recursive bound $\E(D_k^2)\leq Na^{k} + a^2\E(D_{k-1}^2)$ gives 
\begin{align*}
\E(D_k^2) &\;\leq\; \frac{N a^k}{1-a} + a^{2k}N^2.
\end{align*}
Let $P_k = N^{-1} \sum_{i=1}^N X_k^i$, $\tilde P_k = N^{-1} \sum_{i=1}^N Y_k^i$. Clearly $N(P_k - \tilde P_k) \!=\! \sum_{i=1}^N (X_k^i - Y_k^i) \!\leq\! D_k$, and by the above, 
\begin{align}
\label{eq:pnbound}
\E [(P_k - \tilde P_k)^2] \;\leq\; \E(D_k^2)/N^2 \;\leq\; \frac{a^k}{N(1-a)} + a^{2k}.
\end{align}
Write $ \Pi \big( \M_{0,k}(x_0^{1:N}, \uarg) \big)_{q} = \E[\mu_k^{\otimes q}]$ and $\Pi \big( \M_{0,k}(\tilde x_0^{1:N}, \uarg) \big)_{q} = \E[\nu_k^{\otimes q}]$, where
\[
\mu_k = \frac{1}{N} \sum_{i=1}^N M_k(X_{k-1}^i, \uarg), \quad \nu_k = \frac{1}{N} \sum_{i=1}^N M_k(Y_{k-1}^i, \uarg).
\]
We use Lemma \ref{lem:simpler-hellinger} and the inequality $H^2(\mathrm{Ber}(p), \mathrm{Ber}(q)) \leq (p\!-\!q)^2/(4 \min\{p,q,1\!-\!p,1\!-\!q\})$ (see Lemma \ref{lem:h2bernoulli}) to bound
\begin{align*}
\| \E[\mu_k^{\otimes q}] - \E[\nu_k^{\otimes q}]\|_{\mathrm{TV}} \;\leq\;  \sqrt{2q\E H^2(\mu_k, \nu_k)}
\;\leq\; \sqrt{2q \E[(P_{k-1}-\tilde P_{k-1})^2]/(4 \varepsilon)}, 
\end{align*}
where we used $M_k \geq \varepsilon$. Now it follows by \eqref{eq:pnbound} that
\begin{align*}
\sqrt{2q \E[(P_{k-1}-\tilde P_{k-1})^2]/(4\varepsilon)} &\;\leq\; \sqrt{\frac{q a^{k-1}}{N} \frac{ 1}{2\varepsilon (1-a)} + qa^{2k-2} \frac{1}{2 \varepsilon}},
\end{align*}
where the second term tends to zero when $ \delta_\varepsilon \log q - k \to -\infty$ with $\delta_\varepsilon =1/\log((1-2\varepsilon)^{-2})$, and since $qa^{k-1}/N \leq a^{k-1}$, the first term tends to zero as well.
\end{proof}

\begin{remark}
For $q = 1$ we can do the analysis in the previous theorem exactly:
\begin{align*}
 \| \Pi\big( \M_{0,k}(x_0^{1:N}, \uarg)\big)_{q} - \Pi(\M_{0,k}(\tilde x_0^{1:N}, \uarg)\big)_{q} \|_{\mathrm{TV}}  \;=\; | \E(X_k^1) - \E(Y_k^1)|  \;=\; (1-2\varepsilon)^k.
\end{align*}
Hence for any $\delta \in (0,1)$, independently of $N$, $(1-2\varepsilon)^k \leq \delta \iff k \geq 2 \delta_\varepsilon \log(\delta^{-1})$.
\end{remark}

We postpone the proof of the following moment estimates to the appendix.
\begin{lemma}
\label{lem:Pnmoments}
 Define $X_k^{1:N} \sim \M_{0,k}(x_0^{1:N}, \uarg)$ and $Y_k^{1:N} \sim \M_{0,k}(\tilde x_0^{1:N}, \uarg)$, and let $q \in \{1, \ldots, N \}$. For $P_k^q =\frac{1}{q}\sum_{i=1}^q X_k^i $ and $\tilde P_k^q = \frac{1}{q}\sum_{i=1}^q Y_k^i$, it holds that
\begin{gather*}
\E[P_k^q] \;=\; (1-2\varepsilon)^k/2 + \frac12, \quad \E[\tilde{P}_k^q] \;=\; -(1-2\varepsilon)^k/2 + \frac{1}{2}, \quad \text{and} \\
\Var(P_k^q) \;=\; \Var(\tilde{P}_k^q) \;\leq\; \frac{1}{4q} + \frac{1}{4N} \frac{(1-2\varepsilon)^2}{1 - (1-2\varepsilon)^2}.
\end{gather*}
\end{lemma}

\section{Application to out-of-sequence measurements}
\label{sec:oos}
One application of our theory and the coupling methodology discussed above is to the problem of filtering for state-space models in the presence of so-called \emph{out-of-sequence} measurements, which 
are common in real-time
target tracking applications \cite{OrM05,OrG2008}. Out-of-sequence measurements arrive later than other measurements typically due to communication delays. That is, if $y_{\tau}$ is received after a delay $l$, then the observations arrive in the following order: $(y_{\tau+1},\ldots,y_{\tau+l-1},y_{\tau})$.

For ease of exposition, let us assume hereafter that the measurement $y_0$ is delayed. A particle approximation of $\tilde{\eta}_{k+1} = p(x_{k+1}\mid y_{1:k})$ has already been computed, and the missing measurement $y_{0}$ is received late.
This delayed measurement is potentially materially important for the estimation, 
and therefore we wish to compute the particle approximation of $\eta_{k+1} = p(x_{k+1}\mid y_{1:k},y_{0})$.
The problem may be approached by updating the existing filter via 
\[
p(x_{k}\mid y_{1:k},y_{0})\propto p(y_{0}\mid x_{k}, y_{1:k-1})p(x_{k}\mid y_{1:k}).
\]
The problem is that the likelihood $p(y_{0}\mid x_{k}, y_{1:k-1})$ cannot be
computed and strategies have been proposed to approximate it \cite{OrM05,OrG2008}.

The results of the present paper suggest an approach based on a coupled particle filter $(\vec{X}_t,\tilde{\vec{X}}_t)$ where $\vec{X}_t$ targets $\eta_{k+1}$ and $\tilde{\vec{X}}_t$ targets
$\tilde{\eta}_{k+1}$. The filters differ only by their initial potential: $G_0(x)=p(y_0\mid x)$ for the former and $G_0\equiv 1$ for the latter. If the filter states become coupled $\vec{X}_{\sigma}=\tilde{\vec{X}}_{\sigma}$ at some time $\sigma < k+1$, then $\tilde{\vec{X}}_{k+1}=\vec{X}_{k+1}$, and therefore $\tilde{\vec{X}}_{k+1}$ forms a valid particle approximation also for $\eta_{k+1}$. If coupling does not occur, then we adopt $\vec{X}_{k+1}$ as the particle approximation of $\eta_{k+1}$. There are two reasons why coupling may not occur. The first is the sub-optimality of the coupling algorithm, as it is not the actual maximal coupling of $\eta_{k+1}$ and $\tilde{\eta}_{k+1}$. The second is that $\eta_{k+1}$ differs from $\tilde{\eta}_{k+1}$ substantially, that is, the measurement $y_{0}$ is materially important for the estimation.

In practice, the particle system $\tilde{\vec{X}}_0,\ldots,\tilde{\vec{X}}_k$ already exists, and in order to access a coupled filter as discussed above, we need to retrospectively simulate the filter states $\vec{X}_0,\ldots,\vec{X}_{\sigma}$.  Namely, we may use Algorithm \ref{alg:maxcouple-generic} and set $\vec{X}_t \gets \textsc{CondMaxCouple}(\tilde{\vec{X}}_t, \tilde{\mu}_{t}^{\otimes N}, \mu_t^{\otimes N})$ where $\tilde{\mu}_{t}=\Phi_{t}(N^{-1}\sum_{i=1}^{N}\delta_{\tilde{X}_{t-1}^{i}})$ and $\mu_{t}=\Phi_{t}(N^{-1}\sum_{i=1}^{N}\delta_{X_{t-1}^{i}})$, until $\vec{X}_t=\tilde{\vec{X}}_t$, and set $\sigma=t$. As discussed in Section \ref{sec:couplings}, our theory implies that the coupling time $\sigma = O(\log N)$. Therefore, the expected processing cost is bounded independently of the delay $k$, unlike in a naive approach where the particle filter is re-run starting from the delayed measurement.

The method discussed above ensures a valid handling of out-of-sequence measurements in a particle filter.  The coupled system provides also access to approximations of the smoothing distributions. Namely, if a smoothing particle approximation is available at the time of coupling $\sigma$, then a backward smoothing algorithm \cite[see, e.g.][and references therein]{godsill-doucet-west,douc-etal-2011,dau-chopin} can be used to re-calculate the particle approximation of the smoothing distributions before the coupling time $\sigma$, using $\vec{X}_0,\ldots,\vec{X}_{\sigma-1}$.

The procedure detailed above corresponds to the coupled particle filter with the coupling \eqref{eq:state-maximal-coupling}. We may similarly use the individual coupling \eqref{eq:individual-maximal-coupling} or alternate these as discussed in \cite{karjalainen-lee-singh-vihola-mixing}. These may have better practical behaviour, similar to what was observed in the experiments of \cite{karjalainen-lee-singh-vihola-mixing}.
We note that there is a storage cost associated with the method. In practice, the out-of-sequence measurements are
processed until a certain maximum delay to curtail the storage cost \cite{OrG2008}. Measurements which are delayed beyond this maximum are disregarded. The coupling times $\sigma$ can also be used as a diagnostic about the sufficient delay after which out-of-sequence measurements can be safely ignored.

\section{Discussion}\label{sec:discussion}

We established a new result in Theorem \ref{thm:pfforget} on how quickly the particle filter forgets its state in the total variation sense. This is evidently a significant improvement over the result in Lemma \ref{lem:easybound}, which was originally reported in \cite{tadic-doucet-recursive}. We then gave an example in Section \ref{sec:example} where our forgetting rate was shown to be tight. A similar proof technique to Theorem \ref{thm:pfforget} was then used to find the forgetting properties of the conditional particle filter (CPF) in Theorem \ref{thm:cpfforget}, and to derive new propagation of chaos type bounds (Section \ref{sec:pochaos}).

Our result and techniques for understanding the forgetting properties give new theoretical insights into the behaviour of coupled particle filters (Section \ref{sec:couplings}), leading to a justified method for processing out-of-sequence measurements (Section \ref{sec:oos}). We investigate the mixing time of the so-called conditional backward sampling particle filter in a related article \citep{karjalainen-lee-singh-vihola-mixing}, relying on analysis of coupled conditional particle filters. The results of the present paper are essential for this analysis, which provides a sharp mixing time bound (in terms of the time horizon).

In addition to these direct consequences, we think that
the forgetting result for the particle filter could also turn out useful in improving the analysis of the bias of recursive maximum likelihood implemented with a particle filter \cite{tadic-doucet-recursive}. Forgetting could also be a useful tool for the analysis of particle filter stability from a new viewpoint. For instance, our results apply to understanding long-term behaviour of particle filter estimates that are not expectations with respect to the empirical measures.
A similar logarithmic `memory requirement' has also been reported in the context of particle filter variance estimation \cite{mastrototaro2023adaptive}, but it is unclear if there is a deeper connection to be explored. 

Our forgetting result for the CPF is potentially useful in understanding its behaviour over long time horizons and for analysis of smoothing estimates  \cite{cardoso2023state}. Because the CPF can be viewed as a perturbation of the particle filter (see \eqref{eq:perturbselmut} in Section \ref{sec:cpfforget}), we suspect that our analysis could be extended to more general perturbations, which vanish suitably in the number of particles $N$.
These (speculative) extensions are left for future work.

Our proof technique relies heavily on the particles being conditionally independent and identically distributed (see the proof of Theorem~\ref{thm:tvlemmanew}), which means that our result only holds for multinomial resampling. 
For instance, it is not clear how to extend the results to McKean-type models of the second type in \citep[][p. 76]{del2004feynman}, in which the $A_k^i$ in line 3 of Algorithm \ref{alg:pf} are conditionally independent and drawn from a mixture 
$$
A_k^i \sim G_k(X_k^i) \delta_i + \{1-G_k(X_k^i)\}{\rm Categ}(G_k(X_k^1),\ldots,G_k(X_k^N)),
$$
assuming $0 \leq G_k \leq 1$. On the other hand, it is very plausible that similar results would be true in this setting. For example, if the potentials $G_k$ were all the constant function $1$, then the selection step does not occur and the particle filter corresponds to $N$ inhomogeneous Markov chains evolving independently. The time for each chain to forget its initial condition would be upper bounded by a geometric random variable, due to the minorisation (M) in (A1). For all $N$ chains the time to forget would be $O(\log N)$   by considering the distribution of the maximum of $N$ geometric random variables.

Finally, our results rely on a strong mixing assumption (A1), which is common in particle filter theory, but restrictive: in practice, it usually requires a compact state space. 
A more general strong mixing assumption $(M)_m$ of \cite[p.~139]{del2004feynman} assumes instead that \emph{iterates} of $M_k$ satisfy a condition similar to $(M)$, which accommodates also more general proposal distributions \citep{pitt-shephard}. Appendix \ref{app:weaken} discusses how our result can be slightly improved by weakening the \emph{lower bound} assumption for the multi-step density. However, it is not clear if it is possible to prove a similar forgetting result under assumption $(M)_m$ of \cite{del2004feynman} due to the use of Lemma~\ref{lem:h2bound} to bound the squared Hellinger distance in the proof of Theorem~\ref{thm:tvlemmanew}. In order to accommodate mutation kernels that are only regular over several steps, one strategy would be to develop tools to analyse the behaviour of the particle systems over several steps. However, this would seemingly involve a very substantial departure from the techniques adopted in this paper.

It is a natural question whether the forgetting analysis could be extended to accommodate unbounded state spaces, that is, to move significantly beyond (A1).
Our results rely on quantitative stability results (Lemmas \ref{lem:fkcontract} and \ref{lem:pfstability}), which have, to our knowledge, only been proven under strong mixing assumptions. 
There have been some papers that made progress on proving stability results for the particle filter under weaker assumptions (e.g. \citep{van2009uniform,whiteley-stability,douc2014longterm}, see also the discussion in \cite{chopin-papaspiliopoulos} p. 180).
However, the results in these papers provide bounds on asymptotic quantities, in contrast with Lemmas \ref{lem:fkcontract} and \ref{lem:pfstability}, which are non-asymptotic. An alternative approach is taken in \citep{legland2004,oudjane2005stability,crisan2008stability,heine-crisan-uniform}, where the unbounded model is approximated with a model truncated to a compact set, which then satisfies (A1), and they show that sometimes the approximation error is controlled. 
Significant progress on the question of filter stability seems to be a prerequisite for substantially weakening (A1) in the setting of this paper. 

\section*{Acknowledgments}
We thank an anonymous reviewer for a suggestion to investigate low-dimensional marginals, which led to the discussion in Section \ref{sec:qmarginals}.
An AI language model provided by OpenAI was used to proofread the revised manuscript and flag possible typographical, grammatical, and mathematical errors. All suggestions were independently reviewed by the authors, who take full responsibility for the correctness, originality, and integrity of the manuscript.

\section*{Funding}
JK and MV were supported by Research Council of Finland (Finnish Centre of Excellence in Randomness and Structures, grants 346311 \& 364216).
AL was supported by EPSRC grant `CoSInES (COmputational Statistical INference for Engineering and Security)' (EP/R034710/1) and `ProbAI' (EP/Y028783/1).
SSS holds the Tibra Foundation professorial chair and gratefully acknowledges research funding as follows: This material is based upon work supported by the Air Force Office of Scientific Research under award number FA2386-23-1-4100. 

\bibliographystyle{abbrvnat}
\bibliography{refs}

\newpage

\appendix

\section{Technical lemmas}
\label{app:technical}

\begin{lemma}\label{lem:tvbound}
For any probability measures $\mu$ and $\nu$ on $E$ it holds that
\[
\| \mu-\nu \|_{\mathrm{TV}}^2 \leq 1- (1-H^2(\mu, \nu))^2.
\]
\end{lemma}
\begin{proof}
Denote by $H^2(\mu \, \| \,\nu) := 2 H^2(\mu, \nu)$ the (unnormalised) Hellinger divergence. Le Cam's inequality (e.g., \cite{wainwright2019high}) states that
\[
\| \mu-\nu \|_{\mathrm{TV}} \,\leq\, H(\mu \, \| \,\nu) \sqrt{1 - \frac{H^2(\mu \, \| \,\nu)}{4}} \,=\, H(\mu, \nu) \sqrt{2 - H^2(\mu,\nu)},
\]
and the claim follows by squaring both sides and rearranging the terms.
\end{proof}

We note that the following result is similar to Lemma 3.5 in \cite{legland2004}, which gives a relation between expected total variation distances and $L^1$-norms.

\begin{lemma}
\label{lem:h2bound}
Let $\mu$ and $\nu$ be two random probability measures. If \textup{(A1)} holds for a Markov kernel $M$, then
\[
\E H^2(\mu M, \nu M) \leq c'  \sup_{\osc(\phi) \leq 1} \E (| \mu(\phi) - \nu(\phi) |^2),
\]
where $c' = \frac{1}{8}(\bar M / \underline M)^2$.
\end{lemma}
\begin{proof}
For a Markov kernel $M$ with density w.r.t. $\lambda$, $2\mathbb{E}H^{2}(\mu M,\nu M)$ equals

\begin{align*}
\int\mathbb{E}\big( | \sqrt{\mu M(x)}-\sqrt{\nu M(x)} |^{2} \big) \lambda({\rm d}x) &\le\int\mathbb{E}\big(L^2|\mu M(x)-\nu M(x)|^{2}\big) \lambda({\rm d}x)\\
 & \leq L^{2} \int \sup_{\osc(\phi) \leq 1} \mathbb{E} ( \left| \mu(\phi) -\nu(\phi) \right| ^{2} ) \osc(M(\uarg,x))^2 \lambda({\rm d}x)\\
& \leq L^{2} \bar{M}^2  \int  \sup_{\osc(\phi) \leq 1} \E (| \mu(\phi) - \nu(\phi) |^2) \lambda({\rm d}x)\\
&= L^{2} \bar{M}^2 \lambda(E)  \sup_{\osc(\phi) \leq 1} \E (| \mu(\phi) - \nu(\phi) |^2),
\end{align*} 
where the first inequality follows from the lower bound of $M(x,\uarg)\ge\underline{M}>0,$
and where $L=1/\sqrt{ 4\underline M}$ is the Lipschitz constant of the function $x\mapsto\sqrt{x}$
restricted to values $x\ge\underline{M}$.
 For the last part, we have the estimate
\[
1 = M(y,E) = \int M(y,x) \lambda({\rm d}x) \geq \int \underline M \lambda({\rm d}x) = \underline M \lambda(E),
\]
hence $\lambda(E) \leq \underline M^{-1}$, and the result follows with
$c' =  L^{2} \bar{M}^2 \underline M^{-1}/2 =\bar M^2/\underline M^2/8$. \qedhere
\end{proof}

\begin{remark}
As implied by the proof, the constant in Lemma \ref{lem:h2bound} can be (substantially) improved, when $\osc(M(\uarg, x))$ is (much) smaller than $\bar M$. 
\end{remark}

\begin{lemma}
\label{lem:derivative}
Let $b > 1$, and define $f(N) = (1-(1-b/N)^{2N})^{1/2}$. Then, $f'(N)\leq 0$ for $N> b$.
\end{lemma}
\begin{proof}
Define $g(\varepsilon) := (1-\varepsilon^2)^{1/2}$ and $h_b(N) := (1-b/N)^N$.
We observe that $g'(\varepsilon) < 0$ for $\varepsilon \in (0,1)$.
 Moreover, $h_b'(N) \geq 0$ for $N > b$, which follows from
\[
h_b'(N) = (1-b/N)^N \Big( \frac{b}{(1-b/N)N} + \log(1-b/N) \Big)
\]
and the inequality $x/(1+x) \leq \log (1+x)$, $\forall x>-1$, with $x=-b/N$. Since $f(N) = g(h_b(N))$, the result follows.
\end{proof}

\begin{lemma}
  Assume \textup{(A1)}. Then for all $n\geq0$, $N\geq1$, 
  \[
  \mathbb{E}\left[\left\Vert \pi_{n}^{N}M_{n+1}-\pi_{n}M_{n+1}\right\Vert _{{\rm TV}}\right]\leq\frac{c(1)}{2}\cdot\frac{\bar{G}}{\underline{G}}\cdot\frac{\bar{M}}{\underline{M}}\cdot\frac{1}{N^{1/2}},
  \]
where $\pi_n^N = \Psi_n(\eta_n^N)$, $\pi_n = \Psi_n(\eta_n)$, and $c(1)$ is given in Lemma~\ref{lem:pfstability}.
  \label{lem:TV-rootN-upper}
  \end{lemma}
  
  \begin{proof}
  Since $\| \cdot \|_{\mathrm{TV}}$ equals 1/2 times the $L^1$ norm for the densities, we apply Fubini and write
  \begin{align*}
  \mathbb{E}\left[\left\Vert \pi_{n}^{N}M_{n+1}-\pi_{n}M_{n+1}\right\Vert _{{\rm TV}}\right] 
   & \;=\; \frac{1}{2}\int\left\Vert \pi_{n}^{N}M_{n+1}(x)-\pi_{n}M_{n+1}(x)\right\Vert _{1}\lambda({\rm d}x)\\
   & \;=\; \frac{1}{2}\int\left\Vert \pi_{n}^{N}(f_{x})-\pi_{n}(f_{x})\right\Vert _{1}\lambda({\rm d}x)\\
   & \;\leq\; \frac{1}{2}\cdot\frac{\bar{G}}{\underline{G}}\cdot\int\frac{\bar{M}c(1)}{N^{1/2}}\lambda({\rm d}x)
    \;=\; \frac{\lambda(E)}{2}\cdot\frac{\bar{G}}{\underline{G}}\frac{\bar{M}c(1)}{N^{1/2}},
  \end{align*}
  where $f_{x}=z\mapsto M_{n+1}(z,x)$ satisfies ${\rm osc}(f_{x})\leq\bar{M}$
  and we have used Lemma~\ref{lem:pfstability} and Lemma~\ref{lem:bgtransformlp}. Using the estimate $\lambda(E)\leq\underline{M}^{-1}$ as in Lemma~\ref{lem:h2bound},
  we may conclude.
  \end{proof}
  We can complement this general upper bound with a lower bound, which
  for simplicity we state only for simple empirical measures as it is
  sufficient to make the point. In particular, in the setting of the
  following Lemma, we can deduce that the true rate is $O(N^{-1/2})$.
\begin{lemma}
  \label{lem:independent-tv}
Let $\mu$ be a probability measure on $E$, and let $M$ be a transition probability density on $E$ with respect to a measure $\lambda$, such that
\begin{enumerate}
\item[(i)] $\bar{M} := \sup_{x,y} M(x,y) <\infty$,
\item[(ii)] $0 < \int \sqrt{\mathrm{var}_\mu( M(\uarg, x))} \lambda(\ud x) < \infty$.
\end{enumerate}
There exist constants $0 < \underline{c} \le \bar{c} < \infty$ such that
for the empirical measure $\mu^N \!=\! \frac{1}{N} \sum_{k=1}^N \delta_{X^k}$ based on independent $X^1,\ldots,X^N \sim \mu$:
$$
\frac{\underline{c}}{\sqrt{N}} \le 
\E \big[ \| \mu^{N}M-\mu M \|_{{\rm TV}}\big]
\le \frac{\bar{c}}{\sqrt{N}}.
$$
\end{lemma}
\begin{proof}
We may write
\begin{align*}
  \E\left[\left\Vert \mu^{N}M-\mu M\right\Vert _{{\rm TV}}\right] 
    \,=\,\frac{1}{2}\int\mathbb{E}\left[\left|\mu^{N}M(x)-\mu M(x)\right|\right]\lambda({\rm d}x)
   & \,=\,\frac{1}{2}\int\left\Vert \mu^{N}M(x)-\mu M(x)\right\Vert _{1}\lambda({\rm d}x)\\
   & \,=\,\frac{1}{2}\int\left\Vert \mu^{N}(f_{x})-\mu(f_{x})\right\Vert _{1}\lambda({\rm d}x),
\end{align*}
where $f_{x}=z\mapsto M(z,x)$ and $\mathrm{osc}(f_x) \le \bar{M}$. We have
$$
  \left\Vert \mu^{N}(f_{x})-\mu(f_{x})\right\Vert _{2}^2 
    =\frac{{\rm var}_{\mu}(f_{x})}{N} \le \frac{\bar{M}^2}{N},
$$
and by the Marcinkiewicz--Zygmund inequality \citep[e.g.][pp. 498]{shiryaev} and $\mathrm{osc}(f_x) \le \bar{M}$, there is a universal constant $B_4<\infty$ such that
$$
\| \mu^{N}(f_{x})-\mu(f_{x})\|_{4} \le B_4 \Big\| \Big( \sum_{k=1}^N \Big(\frac{f_x(X^k) - \mu(f_x)}{N}\Big)^2\Big)^{1/2} \Big\|_4 \le \frac{B_4\bar{M}}{\sqrt{N}}.
$$
By Hölder's inequality, for any non-negative random variable $Z$,
$\left\Vert Z\right\Vert _{2}^{2} = \| Z^{2/3} Z^{4/3} \|_1 \leq\left\Vert Z\right\Vert _{1}^{2/3}\left\Vert Z\right\Vert _{4}^{4/3}$, which leads to
$$
\left\Vert Z\right\Vert _{2}^{3}\left\Vert Z\right\Vert _{4}^{-2}
\le \left\Vert Z\right\Vert _{1} \le \| Z \|_2.
$$
Applying this with $Z = |\mu^N(f_x) - \mu(f_x)|$ and the bounds above, we obtain
$$
\frac{{\rm var}_{\mu}(f_{x})^{3/2}}{B_4^{2}\bar{M}^{2}}\cdot\frac{1}{\sqrt{N}} \le 
\left\Vert \mu^{N}(f_{x})-\mu(f_{x})\right\Vert _{1} \le 
\frac{{\rm var}_{\mu}(f_{x})^{1/2}}{\sqrt{N}}.
$$
The result follows from assumptions $(i)$ and $(ii)$ as we integrate the lower and upper bounds with respect to $\lambda$.
\end{proof}

\begin{lemma}
\label{lem:tvequality}
Let $E$ be a Polish space. Let $X^{1:N} \sim \mu$, $Y^{1:N} \sim \nu$ for some exchangeable distributions $\mu$, $\nu$ on $E^N$, and define
\[
\eta_x^N = \,  \frac{1}{N} \sum_{i=1}^N \delta_{X^i}, \quad
\eta_y^N = \,  \frac{1}{N} \sum_{i=1}^N \delta_{Y^i}.
\]
Then it holds that $\inf \P(\tilde \eta_x^N \neq \tilde \eta_y^N) = \inf \P(\tilde X^{1:N} \neq \tilde Y^{1:N})$, where the infima are taken over couplings $(\tilde \eta_x^N, \tilde \eta_y^N)$ of $(\eta_x^N, \eta_y^N)$ and $(\tilde X^{1:N}, \tilde Y^{1:N})$ of $(X^{1:N}, Y^{1:N})$, respectively.
\end{lemma}

\begin{proof}
\enquote{$\geq$}: Let $(S_x, S_y)$ be a maximal coupling of $(\eta^N_x, \eta^N_y)$. Then there exist random vectors $B_x$ and $B_y$ on $E^N$ such that
\begin{align*}
\frac{1}{N} \sum_{i=1}^N \delta_{B_x^i} = S_x, \quad \frac{1}{N} \sum_{i=1}^N \delta_{B_y^i} = S_y, 
\end{align*}
and $B_x = B_y$ whenever $S_x = S_y$. 
Let $\pi^{1:N}$ be a uniformly random (and independent of $(B_x, B_y)$) permutation of the vector $(1, \ldots, N)$, and define
\begin{align*}
\bar X^{1:N} &= (B_{x}(\pi^1), \ldots, B_{x}(\pi^N)), \quad
\bar Y^{1:N} = (B_{y}(\pi^1), \ldots, B_{y}(\pi^N)).
\end{align*}
By exchangeability, $(\bar X^{1:N}, \bar Y^{1:N})$ is a coupling of $(X^{1:N}, Y^{1:N})$. From this construction it follows that $\{\bar X^{1:N} = \bar Y^{1:N}\} = \{S_x = S_y \}$, and so $\P(S_x \neq S_y) = \P(\bar X^{1:N} \neq \bar Y^{1:N})$.

\enquote{$\leq$}: This is clear from the fact that for any coupling $(\tilde X^{1:N}, \tilde Y^{1:N})$ of $(X^{1:N}, Y^{1:N})$ it holds that $\{ \tilde X^{1:N} = \tilde Y^{1:N} \} \subset \{\frac{1}{N} \sum_{i=1}^N \delta_{\tilde X^i} = \frac{1}{N} \sum_{i=1}^N \delta_{\tilde Y^i}\}$.
\end{proof}

\begin{proof}[Proof of Lemma \ref{lem:bgtransformlp}]
For probability distributions $\mu$ and $\nu$ one can verify that
\[
\Psi_n(\mu)(\phi) - \Psi_n(\nu)(\phi) = \frac{1}{\mu(G_n)} (\mu - \nu) [G_n(\phi - \Psi_n(\nu)(\phi))],
\]
so that
\[
| \Psi_n(\mu)(\phi) - \Psi_n(\nu)(\phi)|^p = \frac{1}{\mu(G_n)^p} |(\mu - \nu) (\tilde \phi) |^p,
\]
where $\tilde \phi = G_n(\phi - \Psi_n(\nu)(\phi))$. Using the bound $G_n \geq \underline G$ and taking expectations, we find
\begin{equation}
\label{eq:psiineq}
\E( | \Psi_n(\mu)(\phi) - \Psi_n(\nu)(\phi)|^p) \leq \frac{\osc(\tilde \phi)^p}{\underline G ^p} \E( |(\mu - \nu) (\tilde \phi/\osc(\tilde \phi)) |^p).
\end{equation}
It remains to bound $\osc(\tilde \phi)$. Since $\bar G \geq G_n>0$ and $\phi(x) - \Psi_n(\nu)(\phi) \geq 0$ for some $x$,
\[
\sup_x \tilde \phi(x) \leq \bar G \big( \sup_x \phi(x) - \Psi_n(\nu)(\phi) \big),
\]
and since $\phi(x) - \Psi_n(\nu)(\phi) \leq 0$ for some $x$, $\inf_x \tilde \phi(x) \geq \bar G ( \inf_x \phi(x) - \Psi_n(\nu)(\phi) )$.
Thus, $\osc(\tilde \phi) \leq \bar G \osc(\phi)$, and the claim follows from \eqref{eq:psiineq} with $c= \bar G/\underline G$. 
\end{proof}

Lemmas \ref{lem:easybound_cpf}--\ref{lem:tvuniform_cpf} below are used to prove Theorem \ref{thm:cpfforget}, and give CPF analogues to Lemma \ref{lem:easybound}, Theorem \ref{thm:tvlemmanew}, and Lemma \ref{lem:tvuniform}, respectively.
\begin{lemma}
\label{lem:easybound_cpf}
Assume \textup{(A1)}. For all $N \geq 1, n \geq 0, k \geq 1$ and $x^*$,
\[
\beta_\mathrm{TV} (\M_{n,n+k}^{x^*}) \le (1 - \epsilon^N)^k,
\]
where $\epsilon = (\underline M / \bar M)^2$.
\end{lemma}
\begin{proof}
Let $j \geq 1$ and $\mu_{x} := \hat \Phi_{j}( N^{-1}\sum_{i=1}^N \delta_{x^i})$
for arbitrary $x=x^{1:N} \in E^N$, so that $\mu_x^{\otimes N} = \M_j^{x^*_{j-1}}(x, \uarg)$. With this notation established, the rest of the proof proceeds as in Lemma \ref{lem:easybound}.
\end{proof}

\begin{lemma}
\label{lem:tvlemma_cpf}
Assume \textup{(A1)}. There exist $c$ and $c'$, only depending on the constants in \textup{(A1)}, such that for all $n \geq 0$, $N \geq c'$, $k \geq c \log(N)$ and $x^*$,
\[
\beta_{\rm TV}(\M_{n,n+k}^{x^*}) \leq (1-\varepsilon^2)^{1/2},
\]
where $\varepsilon = (1-c'/N)^N$.
\end{lemma}

\begin{proof}
Without loss of generality, assume that $n = 0$. 
For arbitrary $k \geq 2$, $x_0^{1:N}$ and $\tilde x_0^{1:N}$ define
\[
\xi_t = \Phi_{1,t}\Big(\hat \Phi_1 \Big( \frac{1}{N} \sum_{j=1}^N \delta_{x_0^j} \Big) \Big), \quad \tilde \xi_t = \Phi_{1,t}\Big(\hat \Phi_1 \Big( \frac{1}{N} \sum_{j=1}^N \delta_{\tilde x_0^j} \Big) \Big), \quad t = 1, \ldots, k-1.
\]
Let $(X_{t}^{1:N})_{t=0}^{k-2}$ and $(\tilde X_{t}^{1:N})_{t=0}^{k-2}$ be the  particles generated by Algorithm \ref{alg:cpf} using the parameter vectors
\begin{align*}
&(\xi_1 , M_{2:k-1}, G_{1:k-2}, x^*_{1:k-2}, N) \quad \text{and} \quad
(\tilde \xi_1, M_{2:k-1}, G_{1:k-2}, x^*_{1:k-2}, N).
\end{align*}
Let $(\xi_t^N$, $\tilde \xi_t^N)_{t=1}^{k-1} = (N^{-1} \sum_{j=1}^N \delta_{X_{t-1}^j}, N^{-1} \sum_{j=1}^N \delta_{\tilde X_{t-1}^j})_{t=1}^{k-1}$ be the corresponding empirical measures.
We note that $X_{t-1}^{1:N} \sim \M_{0,t}^{x^*} (x_0^{1:N}, \uarg)$ and $\tilde X_{t-1}^{1:N} \sim \M_{0,t}^{x^*} (\tilde x_0^{1:N}, \uarg)$, for $t=1,\ldots,k-1$.

Define $\mu = \hat \Psi_{k-1}(\xi_{k-1}^N)$, $\nu = \hat \Psi_{k-1}(\tilde \xi_{k-1}^N)$, and $M \!=\! M_k$. Then, $\mu M = \hat \Phi_{k}(\xi_{k-1}^N)$,  $\nu M = \hat \Phi_{k}(\tilde \xi_{k-1}^N)$, and
\begin{align*}
\E \left\{ \left(\mu M\right)^{\otimes N} \right\} = \M_{0,k}^{x^*} (x_0^{1:N}, \uarg), \quad 
\E \left\{ \left(\nu M\right)^{\otimes N} \right\} = \M_{0,k}^{x^*}(\tilde x_0^{1:N}, \uarg).
\end{align*}
Bound \eqref{eq:dimbound-2} of Lemma~\ref{lem:dimfree} together with Lemma \ref{lem:simpler-hellinger} gives
\begin{align*}
\big\| \M_{0,k}^{x^*}(x_0^{1:N}, \uarg) - \M_{0,k}^{x^*}(\tilde x_0^{1:N}, \uarg) \big\|_{\mathrm{TV}} &= \left\Vert \E \left\{ \left(\mu M \right)^{\otimes N}\right\} - \E \left\{ \left(\nu M \right)^{\otimes N}\right\} \right\Vert _{\mathrm{TV}} \\
&\leq \sqrt{1-\left(1-\E H^{2}(\mu M, \nu M) \right)^{2N}} \\
&= \sqrt{1-\left(1-\E H^{2}(\hat \Phi_{k}(\xi_{k-1}^N), \hat \Phi_{k}(\tilde \xi_{k-1}^N)) \right)^{2N}}.
\end{align*}
Denote the squared Hellinger distance above by
\[
H^2_{k} := H^2(\hat \Phi_{k}(\xi_{k-1}^N), \hat \Phi_{k}(\tilde \xi_{k-1}^N)).
\]
Applying Lemma~\ref{lem:h2bound} to $\hat \Phi_{k}(\xi_{k-1}^N) = \hat \Psi_{k-1}(\xi_{k-1}^N)M_{k}$ and $\hat \Phi_{k}(\tilde \xi_{k-1}^N) = \hat \Psi_{k-1}(\tilde \xi_{k-1}^N)M_{k}$  yields
\[
\E H^2_{k} \leq C \sup_{\osc(\phi) \leq 1} \E (|\hat \Psi_{k-1}(\xi_{k-1}^N)(\phi) - \hat \Psi_{k-1}(\tilde \xi_{k-1}^N)(\phi)|^2),
\]
where $C = (1/8) (\bar M/ \underline M)^2$. 
Denote $\zeta_n = \Psi_n (\xi_{n})$, $\tilde \zeta_n = \Psi_n (\tilde \xi_{n})$. Then
\begin{align*}
\| \hat \Psi_{k-1} (\xi_{k-1}^N)(\phi) - \hat \Psi_{k-1} (\tilde  \xi_{k-1}^N)(\phi) \|_2 \;\leq \;&\;\| \hat \Psi_{k-1} ( \xi_{k-1}^N)(\phi) - \zeta_{k-1} (\phi) \|_2  
 \,+ |\zeta_{k-1} (\phi) - \tilde \zeta_{k-1} (\phi) | \\
\,+ \;&\;\| \hat \Psi_{k-1}(\tilde{ \xi}_{k-1}^N)(\phi) - \tilde \zeta_{k-1}(\phi) \|_2,
\end{align*}
where Theorem \ref{thm:cpfstability} gives the bounds
\begin{align*}
\| \hat \Psi_{k-1} ( \xi_{k-1}^N)(\phi) - \zeta_{k-1} (\phi) \|_2 \leq \frac{C_2}{\sqrt N}, \qquad 
\| \hat \Psi_{k-1}(\tilde{ \xi}_{k-1}^N)(\phi) - \tilde \zeta_{k-1}(\phi) \|_2 \leq \frac{C_2}{\sqrt N}, 
\end{align*}
and by Lemmas \ref{lem:bgtransformlp} and \ref{lem:fkcontract},
\[
|\zeta_{k-1} (\phi) - \tilde \zeta_{k-1} (\phi) | \leq \sup_{\mu, \nu} \| \Psi_{k-1} (\Phi_{1,k-1}(\mu)) - \Psi_{k-1}( \Phi_{1,k-1}(\nu)) \|_{\mathrm{TV}} \leq C_3 \beta^{k-2}
\]
with $C_3 = \bar G / \underline G$ and $\beta = 1-(\underline M /\bar M)^2$. If $ k \geq c \log(N)$ with 
\[
c = \frac{1}{2 \log (\beta^{-1})} + \frac{2}{\log 2},
\]
then $k -2 \geq \log(N)/(2\log(\beta^{-1}))$ for all $N\geq 2$. Since $C_2 \geq C_3$ by \eqref{eq:cpfconstant}, we have $C_3 \beta^{k-2} \leq C_2/\sqrt{N}$ and
\[
\| \hat \Psi_{k-1} (\xi_{k-1}^N)(\phi) - \hat \Psi_{k-1} (\tilde  \xi_{k-1}^N)(\phi) \|_2 \leq 3\frac{C_2}{\sqrt N},
\] 
and we conclude that $\E H^2_{k} \leq c' N^{-1}$ with $c' = 9C C_2^2$.
\end{proof}

\begin{lemma}
\label{lem:tvuniform_cpf}
Assume \textup{(A1)}. There exist $N_{\rm min} \geq 2$, $c>1$,  and $\alpha \in [0,1)$, only depending on the constants in \textup{(A1)}, such that for all $n \geq 0$, $N \geq N_{\rm min}$, $k \geq c \log(N)$, and $x^*$,
\[
\beta_\mathrm{TV}(\M_{n,n+k}^{ x^*} ) \leq \alpha.
\]
\end{lemma}
\begin{proof}
Follows immediately from Lemmas \ref{lem:tvlemma_cpf} and \ref{lem:derivative} with $c$ and $c'$ as in Lemma \ref{lem:tvlemma_cpf}, $N_{\rm min} = \lfloor c' \rfloor +1$, and
\[
\alpha = \big(1-(1-c'/N_{\rm min})^{2 N_{\rm min}}\big)^{1/2}. \qedhere
\]
\end{proof}

\begin{proof}[Proof of Lemma \ref{lem:rknbound}]
Denote $g = Q_{k-1,n}(G_n)$, $h = P_{k-1,n}(\phi)$. From the definition \eqref{eq:R_k} of $R_k$,
\[
R_k = \frac{(N \hat \eta_{k-1}^N + \delta_{x_{k-1}^*})(gh)}{(N \hat \eta_{k-1}^N + \delta_{x_{k-1}^*})(g)} - \frac{\hat \eta_{k-1}^N(gh)}{\hat \eta_{k-1}^N (g)},
\]
where we used the fact $\Phi_k(\mu)(f)=\mu Q_{k-1,k}(f)/\mu Q_{k-1,k}(1)$. By the definition of $\gamma_{k,n}$ in Theorem \ref{thm:lpbound_cpf},
\[
\frac{a+b}{c+d} - \frac{a}{c} = \frac{d}{c+d}\Big(\frac{b}{d} - \frac{a}{c}\Big) \quad \Rightarrow \quad R_k = \gamma_{k-1,n} \Big( h(x_{k-1}^*) - \frac{\hat \eta_{k-1}^N (gh)}{\hat \eta_{k-1}^N (g)} \Big).
\]
Since $\hat \eta_{k-1}^N (gh)/ \hat \eta_{k-1}^N (g)$ is between $\inf h$ and $\sup h$ as a weighted mean, we find
\[
|R_k| \leq \gamma_{k-1,n} \osc(h). \qedhere
\]
\end{proof}

\begin{proof}[Proof of Lemma \ref{lem:Pnmoments}]
Define $\Y_k^q := P_k^q - 1/2$ and $\tilde\Y_k^q := \tilde P_k^q - 1/2$. Using the definition of $\M_k$, it follows that
\begin{align*}
\E[\Y_k^q\mid \Y_{k-1}^N] \;=\; \E[P_k^q\mid P_{k-1}^N] - 1/2  \;=\; \varepsilon + (1-2\varepsilon) P_{k-1}^N - 1/2 
&\;=\; (1-\varepsilon) \Y_{k-1}^N - \varepsilon \Y_{k-1}^N \\
&\;=\; (1-2\varepsilon) \Y_{k-1}^N,
\end{align*}
and since $\Y_0^N = 1/2$, we conclude that $\E[\Y_k^q] = (1-2\varepsilon)^k/2 = \E[P_k^q] -1/2$.
The formula for $\E[\tilde P_k^q]$ follows from $\tilde \Y_0^N = -1/2$ by the same argument. Since $q P_k^q \, \vert \, P_{k-1}^N \sim \text{Binomial}(q, \, 1/2 + (1-2\varepsilon) \Y_{k-1}^N )$,
\begin{align*}
\Var(\Y_k^q\mid \Y_{k-1}^N) \;=\; \frac{1}{q^2}\Var(qP_k^q\mid P_{k-1}^N) 
&\;=\; \frac{1}{q} \big(1/2 + (1-2\varepsilon) \Y_{k-1}^N \big)
\big(1/2 - (1-2\varepsilon) \Y_{k-1}^N \big) \\
&\;=\; \frac{1}{q} \Big( \frac{1}{4} - 
(1-2\varepsilon)^2 (\Y_{k-1}^N)^2 \Big) \;\le\; \frac{1}{4q}.
\end{align*}
For $q=N$, the upper bound follows from the recursive relation
\begin{align*}
\Var(\Y_k^N) \;=\; \E[\Var(\Y_k^N \mid \Y_{k-1}^N)] + \Var(\E[\Y_k^N\mid \Y_{k-1}^N]) 
&\;\le\; \frac{1}{4N} + (1-2\varepsilon)^2 \Var(\Y_{k-1}^N) \\
&\;\le\; \frac{1}{4N} \sum_{i=0}^{k-1} (1-2\varepsilon)^{2i} 
\;\le\; \frac{1}{4N} \frac{1}{1 - (1-2\varepsilon)^2},
\end{align*}
where the second inequality followed from $\Var(\Y^N_0) = 0$. For $q < N$, we conclude by
\[
\Var(P_k^q) \;=\; \E[\Var(P_k^q \mid \Y_{k-1}^N)] + \Var(\E[P_k^q \mid \Y_{k-1}^N]) 
\;\le\; \frac{1}{4q} + (1-2\varepsilon)^2 \Var(\Y_{k-1}^N). \qedhere
\]
\end{proof}

\begin{lemma}
\label{lem:h2bernoulli}
For $p \in (0,1)$, $q \in (0,1)$,
\[
H^2(\mathrm{Ber}(p), \mathrm{Ber}(q)) \leq \frac{(p-q)^2}{4 \min\{p,q,1-p,1-q\}}.
\] 
\end{lemma}
\begin{proof}
By definition, 
$H^2(\mathrm{Ber}(p), \mathrm{Ber}(q)) =  ( (\sqrt p - \sqrt q)^2 + (\sqrt{1-p} - \sqrt{1-q})^2 )/2$,
and since $|p - q| = |\sqrt p + \sqrt q| |\sqrt p -\! \sqrt q| \;\geq\; 2 \min\{ \sqrt p, \sqrt q \} |\sqrt p -\! \sqrt q|$, it follows that
\[
(\sqrt p - \!\sqrt q)^2 \leq \frac{(p-q)^2}{4 \min\{p,q\}}, \quad (\sqrt{1-p} - \sqrt{1-q})^2 \leq \frac{(p-q)^2}{4 \min\{1-p,1-q\}}. \qedhere
\]
\end{proof}

\section{On weakening the assumptions (A1)}
\label{app:weaken}

We will now discuss how Theorem \ref{thm:pfforget} can be shown to hold if we replace Assumption \textup{(A1)} with \textup{(A2)} below -- in particular, we no longer assume lower bounds on the one-step transition kernels $M_t$. This type of situation can arise, e.g., in models where it is natural to restrict the kernels $M_t(x, y)$ so that transitions are only allowed to points $y$ within some maximum distance from $x$, but each point in $E$ is reachable given sufficiently many time steps.
\begin{assumption*}[A2]
There exist constants $\underline G>0$, $\Mm>0$, $\bar G < \infty$,  $\bar M < \infty$ such that assumption \textup{A1}$(G)$ holds with $\underline G$, $\bar G$, and
\begin{itemize}
\item $(M^+):$ The Markov kernels have densities such that $M_t \leq \bar M $ for all $t$, with respect to some finite measure $\lambda$.
\item $(M_m):$ There exists $m \geq 1$ such that the compositions $M_{t,t+m} = M_{t+1} \ldots M_{t+m}$ have densities satisfying $ \Mm \leq  M_{t,t+m}(x,y) $ for all $t$.
\end{itemize}
\end{assumption*}

A crucial part of the proof of Theorem \ref{thm:pfforget} is Lemma \ref{lem:h2bound}, which leads to the decomposition of $L^2$ errors in \eqref{eq:tvlemnewproof-2}. We discuss below how we can obtain a similar result (Lemma \ref{lem:h2boundalt}) with (A2). Because (A2) implies time-uniform errors (Remark \ref{rem:alt-time-uniform}), we conclude that a forgetting result analogous to Theorem \ref{thm:pfforget} holds with (A2).

\begin{remark}
  \label{rem:alt-time-uniform}
Assumption (A2) $(M^+)$ implies that $M_{t,t+m}(x,y) \le \bar{M}$, and therefore the condition $(M)_m$ of \cite[p.~139]{del2004feynman} holds. Proposition 4.3.4 and Theorem 7.4.4 in \cite{del2004feynman} give analogues of Lemmas \ref{lem:fkcontract} and \ref{lem:pfstability} under (A2):
\begin{align}
\sup_{\mu, \nu} \| \Phi_{n,n+k}(\mu) - \Phi_{n,n+k}(\nu) \|_\mathrm{TV} &\leq \beta_m^{\lfloor k/m \rfloor}, \label{eq:alt-ideal-forgetting}\\
\sup_{n\geq0} \sup_{\osc(\phi) \leq 1} \| \eta_n^N(\phi) - \eta_n(\phi) \|_p &\leq \frac{c_{p,m}}{\sqrt{N}}, \label{eq:alt-uniform-errors}
\end{align}
for some $\beta_m < 1, c_{p,m} < \infty$, which only depend on $p$, $m$, and the constants in \textup{(A2)}. 
\end{remark}

Under \textup{(A2)}, we obtain a lower bound for the densities of the ideal filters for large enough time indices, not depending on the initial distribution $\eta_0$. We will use this fact in the proof of Lemma \ref{lem:h2boundalt} to conclude that, despite the lack of lower bound on $M_t$, the particle filter is 'sufficiently unlikely' to reach low-density areas with large $N$. Here and hereafter we use $\eta_k(x)$ to denote a density of $\eta_k$ w.r.t. the measure $\lambda$ in \textup{(A2)} (and similarly for $\tilde \eta_k(x)$, $\tilde \eta_k$).

\begin{lemma}
\label{lem:etabound}
Assume \textup{(A2)} for some $m\geq 1$. Then there exists a constant $\epsilon_m>0$, such that for all $k\geq m$ there exists a density $\eta_k(x) \geq \epsilon_m$. In particular, we can choose $\epsilon_m = (\underline G/ \bar G)^m \Mm$.
\end{lemma}
\begin{proof}
We use the fact that for any $k \geq m$, 
\[
\eta_k = \Phi_{k-m,k} ( \eta_{k-m} ).
\]
Denoting $\E_{p,\mu}(f_{p}) = \int \int f_{p}(x_p, \ldots, x_k) M_{p+1}(x_p,\ud x_{p+1}) \ldots M_k(x_{k-1},\ud x_k) \mu(\ud x_p)$, we can write for any distribution $\mu$ and bounded test function $f\geq 0$ on $E$: 
\begin{align*}
\Phi_{k-m,k}(\mu)(f) &\;=\; \frac{\E_{k-m,\mu} (f(X_k) \prod_{q=k-m}^{k-1} G_q(X_q))}{\E_{k-m,\mu} ( \prod_{q=k-m}^{k-1} G_q(X_q))} 
\;\geq\; \Big(\underline G/ \bar G \Big)^m \E_{k-m,\mu} (f(X_k)),
\end{align*}
and
\begin{align*}
\E_{k-m,\mu} (f(X_k)) &\;=\; \int \int f(x_k) M_{k-m,k}(x_{k-m},x_{k} ) \lambda(\ud x_k) \mu(\ud x_{k-m}) \;\geq\; \int f(x_k) \Mm \lambda(\ud x_k).
\end{align*}
Since this is true for arbitrary $\mu$, we conclude.
\end{proof}

Lemma \ref{lem:h2boundalt} below gives an alternative way to bound $\E H_k^2$ in the proof of Theorem \ref{thm:tvlemmanew}, bypassing the need to use Lemma \ref{lem:h2bound} (and therefore Assumption \textup{(A1)}). We will use the following notation: Let $X^{1:N}_{k-1} \sim \M_{0,k-1}(x_0^{1:N}, \cdot)$, $\tilde X_{k-1}^{1:N} \sim \M_{0,k-1}(\tilde x_0^{1:N}, \cdot)$ with arbitrary initial states 
$x_0^{1:N}$ and $\tilde x_0^{1:N}$. Define the corresponding (updated) empirical measures
\[
\pi_{k-1}^N = \Psi_{k-1} \Big(\frac{1}{N} \sum_{j=1}^N \delta_{X_{k-1}^j} \Big), \quad \tilde \pi_{k-1}^N = \Psi_{k-1} \Big(\frac{1}{N} \sum_{j=1}^N \delta_{\tilde X_{k-1}^j} \Big),
\]
the ideal filters
\[
\eta_k = \Phi_{0,k} \Big( \frac{1}{N} \sum_{j=1}^N \delta_{x_{0}^j} \Big), \quad \tilde \eta_k = \Phi_{0,k} \Big( \frac{1}{N} \sum_{j=1}^N \delta_{\tilde x_{0}^j} \Big),
\]
and the updated versions $\pi_k = \Psi_k(\eta_{k})$, $\tilde \pi_k = \Psi_k(\tilde \eta_{k})$.

\begin{lemma}
\label{lem:h2boundalt}
Assume \textup{(A2)} for some $m\geq 1$ and denote $D(\mu,\nu) = \sup_{\osc(\phi) \leq 1} \|\mu(\phi) - \nu(\phi) \|_2^2$. Then, for all $k \geq m$,
\begin{align*}
\E H^2(\pi_{k-1}^N M_k, \tilde \pi_{k-1}^N M_k) \leq C \lambda(E) \Big(  D( \pi_{k-1}^N, \tilde \pi_{k-1}^N) 
+ D(\pi_{k-1}^N, \pi_{k-1}) 
+ D(\tilde \pi_{k-1}^N, \tilde \pi_{k-1}) \Big),
\end{align*}
where we can choose 
$C = 2 \bar M^2 (\bar G / \underline G)^m \Mm^{-1} \max \Big\{\frac{1}{8}, \bar M (\bar G / \underline G)^m \Mm^{-1} \Big\}$.
\end{lemma}

\begin{remark}
  By \eqref{eq:alt-uniform-errors} in Remark \ref{rem:alt-time-uniform} and Lemma \ref{lem:bgtransformlp}, the last two terms in the sum above are of the order $O(1/N)$, and $\| \pi_{k-1}^N(\phi) - \tilde \pi_{k-1}^N (\phi) \|_2$ is bounded by
\[
\| \pi_{k-1}^N(\phi) - \pi_{k-1} (\phi) \|_2 + \| \tilde \pi_{k-1}^N(\phi) - \tilde \pi_{k-1} (\phi) \|_2 + \| \pi_{k-1}(\phi) - \tilde \pi_{k-1} (\phi) \|_2,
\]
where the last term vanishes exponentially by the forgetting property of the ideal filter \eqref{eq:alt-ideal-forgetting}. This leads to the desired order of $O(1/N)$ for $\E H^2(\pi_{k-1}^N M_k, \tilde \pi_{k-1}^N M_k)$ for large $k$.
\end{remark}

\begin{proof}[Proof of Lemma \ref{lem:h2boundalt}]
Denote $M = M_k$, $\mu = \pi_{k-1}^N$, and $\tilde \mu = \tilde \pi_{k-1}^N$.
By Lemma \ref{lem:etabound}, we can choose $\epsilon>0$ such that $\eta_k(x), \tilde \eta_k(x) \geq 2\epsilon$, $\forall x$. Define now $B_\epsilon(x) = \{ \mu M(x) < \epsilon \}$,  $\tilde{B}_\epsilon(x) = \{ \tilde{\mu} M(x) < \epsilon \}$ (the 'low-density sets') and $G_\epsilon(x) = B_\epsilon^C(x) \cap \tilde{B}_\epsilon^C(x)$ (the 'high density set'). We split the state space into two parts and bound the expected $H^2$:
\begin{align*}
2 \E H^2(\mu M, \tilde{\mu} M)
& = \int \E \bigg[ \big(\I(G_\epsilon(x)) + \I(G_\epsilon(x)^C) \big) \Big(\sqrt{\mu M(x)} - \sqrt{\tilde{\mu} M(x)}\Big)^2 \bigg]\lambda(\ud x )\\
& \le \frac{1}{4\epsilon} 
\int \E\bigg[  \Big(\mu M(x) - \tilde{\mu} M(x)\Big)^2 \bigg] \lambda(\ud x) + \bar{M}\lambda(E)\sup_x \P\big(G_\epsilon^C(x)\big) \\
& \le \frac{\lambda(E)}{4\epsilon} 
\sup_x \E\bigg[  \Big(\mu M(x) - \tilde{\mu} M(x)\Big)^2 \bigg]  + \bar{M}\lambda(E)\sup_x\P\big(G_\epsilon^C(x)\big),
\end{align*}
where we used the Lipschitz continuity of $x\mapsto \sqrt{x}$ restricted to $x\ge \epsilon$, with constant $(4\epsilon)^{-1/2}$. Clearly,
$$
  \P\big(G_\epsilon^C(x)\big)
  = \P \big( B_\epsilon(x) \cup \tilde{B}_\epsilon(x)\big)
  \le \P(B_\epsilon(x)) + \P(\tilde{B}_\epsilon(x)),
$$
and both probabilities can be bounded similarly: 
\begin{align*}
    \P(B_\epsilon(x)) \,=\, 
    \P( \mu M(x) < \epsilon, \eta_k(x) \ge 2\epsilon) 
    & \,\le\, \P(\eta_k(x) - \mu M(x) \ge \epsilon) \\
    &\,\le\, \P\big( | \mu M(x) - \eta_k (x) | \ge \epsilon\big) \\
    &\,\le\, \frac{\E | \mu M(x) - \eta_k (x) |^2}{\epsilon^2} \\
    &\,\le\, \frac{1}{\epsilon^2} \mathrm{osc} \big(M(\;\cdot\;, x)\big)^2 \sup_{\mathrm{osc}(\phi)\le 1} \E | \mu (\phi) - \pi_{k-1} (\phi) |^2.
\end{align*}
Now we can apply the same argument to $\P(\tilde B_\epsilon(x))$, and bound the remaining term in the decomposition:
$$
 \frac{\lambda(E)}{4\epsilon} \sup_x \E\Big[  \big(\mu M(x) - \tilde{\mu} M(x)\big)^2 \Big]
\;\le\;  \frac{\lambda(E)}{4\epsilon}  \sup_x \mathrm{osc}(M(\;\cdot\;,x))^2 \sup_{\mathrm{osc}(\phi)\le 1} \E | \mu(\phi) - \tilde{\mu}(\phi) |^2.
$$
It remains to fix $\epsilon$ and bound the constant $C$ in the claim. Using the fact that $\mathrm{osc}(M(\;\cdot\;, x)) < \bar M$ and choosing $\epsilon = \epsilon_m/2 = (\underline G/ \bar G)^m \Mm/2$ (which satisfies $\eta_k(x), \tilde \eta_k(x) \geq 2\epsilon$ by Lemma \ref{lem:etabound}), we find
\begin{align*}
\frac{1}{2} \max \Big\{\frac{1}{4\epsilon} \mathrm{osc}(M(\;\cdot\;, x))^2, \;\frac{1}{\epsilon^2}\bar M \mathrm{osc}(M(\;\cdot\;, x))^2  \Big\} &\;\leq\; \frac{1}{2}\bar M^2 \epsilon^{-1} \max \Big\{\frac{1}{4}, \; \bar M \epsilon^{-1} \Big\} \\
&\;=\; \bar M^2 \epsilon_m^{-1} \max \Big\{\frac{1}{4}, \; 2 \bar M \epsilon_m^{-1} \Big\}.
\qquad \qedhere
\end{align*}
\end{proof}

\end{document}